\documentclass[11pt,leqno]{amsart}
\usepackage[%
	a4paper,
	foot=20pt,
	top=80pt,
	headsep=20pt,
	left=84pt,
	right=84pt,
	driver=dvipdfm
]{geometry}
\usepackage{cite}
\usepackage{amsmath,amssymb,amsthm}
\usepackage{txfonts}
\usepackage[dvipdfmx]{graphicx}
\usepackage{subfigure}
\usepackage[inline]{enumitem}
\setlist[enumerate,1]{%
	label={\normalfont(\arabic*)},ref=\arabic*%
}
\def\setmid{\mathrel{}\middle\vert\mathrel{}}%
\def\card#1{\lvert{#1}\rvert}
\def\transp#1{{}^{\mathrm{t}}#1}
\def\dtriangle{\rotatebox[origin=c]{180}{$\triangle$}}
\def\cm{\textsc{m}}
\DeclareMathOperator{\GC}{GC}

\DeclareMathOperator{\sgn}{sgn}
\DeclareMathOperator{\id}{id}

\theoremstyle{definition}
\newtheorem{Def}{Definition}[section]

\newtheorem{Exs}[Def]{Examples}
\theoremstyle{plain}
\newtheorem{Thm}[Def]{Theorem}
\newtheorem{Lem}[Def]{Lemma}
\newtheorem{Prop}[Def]{Proposition}

\theoremstyle{remark}
\newtheorem{Rem}[Def]{Remark}
\newtheorem{Rems}[Def]{Remarks}
\makeatletter
\@addtoreset{equation}{section}

\makeatother
\allowdisplaybreaks
\author{Toshiaki Omori, Hisashi Naito, and Tatsuya Tate}
\address{%
T.~Omori:
Cifra Co., Ltd., Tokyo 135-0047, Japan
}
\email{tosiaki.omori@gmail.com}
\address{%
H.~Naito:
Graduate School of Mathematics, 
Nagoya University, 
Chikusa, Nagoya 464-8602, Japan
}
\email{naito@math.nagoya-u.ac.jp}
\address{%
T.~Tate: 
Mathematical Institute,
Tohoku University, 
Aoba, Sendai 980-8578, Japan
}
\email{tatsuya.tate.c6@tohoku.ac.jp}
\title{Eigenvalues of the Laplacian 
on the Goldberg-Coxeter constructions 
for $3$- and $4$-valent graphs}
\usepackage{fancyhdr} 

\begin{document}
\begin{abstract}
We are concerned with spectral problems 
of the Goldberg-Coxeter construction 
for $3$- and $4$-valent finite graphs. 
The Goldberg-Coxeter constructions 
$\GC_{k,l}(X)$ of a finite $3$- or $4$-valent graph $X$ 
are considered as ``subdivisions'' of $X$, 
whose number of vertices are increasing 
at order $O(k^2+l^2)$, 
nevertheless which have bounded girth. 
It is shown that the first (resp.\ the last) 
$o(k^2)$ eigenvalues of the combinatorial Laplacian 
on $\GC_{k,0}(X)$ tend to $0$ (resp.\ tend to $6$ or $8$ 
in the $3$- or $4$-valent case, respectively) 
as $k$ goes to infinity. 
A concrete estimate for the first several eigenvalues 
of $\GC_{k,l}(X)$ by those of $X$ is also obtained 
for general $k$ and $l$. 
It is also shown that the specific values 
always appear as eigenvalues of $\GC_{2k,0}(X)$ 
with large multiplicities 
almost independently to the structure of the initial $X$. 
In contrast, 
some dependency of the graph structure of $X$ 
on the multiplicity of the specific values 
is also studied. 
\end{abstract}
\maketitle
\thispagestyle{empty}
%
\section{Introduction}
\label{Sec(Intro)}
The \emph{Goldberg-Coxeter construction} 
is a subdivision of a $3$- or $4$-valent graph, 
and it is defined by Dutour-Deza~\cite{MR2429120} 
for a plane graph based on a simplicial subdivision 
of regular polytopes in \cite{MR0448238,1937104}. 
In \cite{MR2429120}, it is pointed out that 
it often appears in chemistry and architecture, 
and its combinatorial and algebraic structures 
are investigated. 
Goldberg-Coxeter constructions of regular polyhedra 
generate a class of Archimedean polyhedra, 
and infinite sequence of polyhedra, 
which are called Goldberg polyhedra. 
For example a Goldberg-Coxeter construction 
of a dodecahedron generates a truncated-icosahedron, 
which is known as a fullerene $C_{60}$ 
\cite{ISI:000254695000004,ISI:000332180900020}. 
Goldberg-Coxeter constructions are also applied 
to Mackay-like crystals, 
and explain large scale of spatial fullerenes 
\cite{MR3724208,MR3557605}. 
Mathematical modeling of self-assembly in nature 
is also widely studied in 
\cite{MR0448238,ISI:000342904500001}. 
Recently, Fujita et al.\ have synthesized 
molecule structures with $4$-valent Goldberg polyhedra, 
and they explain self-assembly from viewpoints of 
chemistry and biology \cite{ISI:000391190500047}. 
\par
On the other hand, 
the stability of a molecule is explained 
by eigenvalues of the finite graphs 
which express the molecule structure 
by H\"uckel method \cite{MR2571608}. 
Hence, studies for eigenvalues 
of Goldberg-Coxeter constructions are worth trying. 
DeVos, Goddyn, Mohar, and \v{S}\'{a}mal considered 
eigenvalues of $(3,6)$-fullerenes (cf.\mbox{\cite{MR2482954}}). 
A $(3,6)$-fullerene $X$ is a $3$-valent plane graph 
whose faces are triangles and hexagons. 
We consider general $3$- or $4$-valent graphs, 
and if a graph is $(k,6)$-fullerene, 
then its Goldberg-Coxeter constructions 
are also $(k,6)$-fullerenes. 
The Goldberg-Coxeter construction $\GC_{k,l}(X)$ 
of a $3$- or $4$-valent graph $X$ 
has the parameters $k$ and $l$ both of which are integers 
and they are regarded to indicate a point 
in the triangular or square lattices, respectively. 
Then we are concerned with behavior 
of eigenvalues of $\GC_{k,l}(X)$ 
when $k$ and $l$ tend to infinity. 
\par
Throughout this paper, unless otherwise indicated, 
a graph is always assumed to be connected, 
finite and simple. 
For a graph $X$, let us denote 
by $V(X)$ the set of vertices of $X$, 
and by $E(X)$ the set of undirected edges of $X$. 
For $p\in V(X)$, the set of its neighboring vertices 
is denoted by $N_X(p)$. 
The combinatorial Laplacian $\Delta_X$, 
simply called the \emph{Laplacian}, 
of a graph $X$ acts on the set $\mathbb{C}^{V(X)}$ 
of functions on $V(X)$ and is defined as
\[
	(\Delta_Xf)(p)
	:=
	\deg(p)f(p)
	-\sum_{q\in N_X(p)}f(q)
	\quad
	\text{for $f\in \mathbb{C}^{V(X)}$ and $p\in V(X)$},
\]
where $\deg(p)$ denotes the degree of the vertex $p$.
As is well-known, the eigenvalues of $\Delta_X$ 
for a regular graph $X$ of degree $r$ 
necessarily lie in the interval $[0,2r]$. 
\par
The definition of the Goldberg-Coxeter constructions 
extends for general $3$- or $4$-valent graph 
$X = (V (X), E(X))$ equipped with 
an orientation at each vertex, 
in the sense that, for each $p\in V(X)$, 
the set of edges emanating from $p$ is ordered. 
As shall be explained later 
(cf.\ Proposition~\ref{Prop(embedded_on_surface)}), 
if, in particular, $X$ is ``appropriately'' embedded 
in an oriented surface, then 
$X$ is endowed with a natural orientation at each vertex 
and $\GC_{k,l}(X)$ remains to be also embedded 
in the same surface. 
\par
There is a long line of works on upper bounds 
for the (especially, first nonzero) eigenvalues 
of general planar or genus $g$ finite graphs 
(see \mbox{\cite{MR2203731,MR2294342}} 
and the references therein). 
In \mbox{\cite{MR2846385}}, it is proved that 
the $i$-th eigenvalue of a graph embedded in 
an oriented surface of genus $g$ is estimated from above by 
$O((g+1)\log^2(g+1)i/n)$, 
where $n$ is the number of the vertices. 
The following theorem 
does not only depend on the genus, 
but contains an assertion 
on the last several eigenvalues of $\GC_{k,0}(X)$. 
\begin{Thm}
\label{Thm(o(k^2)_evs)}
Let $X=(V(X),E(X))$ be a connected, finite and simple 
$3$- or $4$-valent graph 
equipped with an orientation at each vertex, and 
$\GC_{k,0}(X)$ be the Goldberg-Coxeter construction of $X$ 
for $k\geq 1$. 
Then, for any number $o(k^2)$ satisfying $o(k^2)/k^2\to 0$ 
as $k\to \infty$, 
the first {\upshape(}resp.\ the last{\upshape)} 
$o(k^2)$ eigenvalues of the Laplacian of\/ $\GC_{k,0}(X)$ 
tend to $0$ {\upshape(}resp.\ tend to $6$ or $8$ 
in the $3$- or $4$-valent case, respectively{\upshape)} 
as $k$ goes to infinity. 
\end{Thm}
Here we note that $\GC_{k,0}(X)$ has 
$k^2\card{V(X)}=O(k^2)$ vertices and 
the above result is the best in matters of the convergence 
to $0$ or to the natural upper bound. 
As for the first and the last $\card{V(X)}$ eigenvalues, 
the following concrete estimates are also obtained.
\begin{Thm}
\label{Thm(comp_with_X)}
Let $X=(V(X),E(X))$ be a connected, finite and simple 
$3$- or $4$-valent graph 
equipped with an orientation at each vertex, 
$X'=\GC_{k,l}(X)$ 
be the Goldberg-Coxeter construction of $X$, 
where $k\geq l\geq 0$ and $k\neq 0$ and
\begin{gather*}
	0=\lambda_1(X)<\lambda_2(X)\leq \cdots \leq 
	\lambda_{\card{V(X)}}(X),
	\\
	0=\lambda_1(X')<\lambda_2(X')\leq \cdots \leq 
	\lambda_{\card{V(X')}}(X')
\end{gather*}
be the eigenvalues of their Laplacians 
$\Delta_X$, $\Delta_{X'}$, respectively. 
Then for $i=1,2,\dots,\card{V(X)}$,
\begin{equation}
	\lambda_i(\GC_{k,l}(X))
	\leq
	\left\{
	\begin{aligned}
		&\frac{3k}{k^2+kl+l^2}\lambda_i(X),
		& & \text{if $X$ is $3$-valent},
		\\
		&\frac{2k}{k^2+l^2}\lambda_i(X),
		& & \text{if $X$ is $4$-valent}.
	\end{aligned}
	\right.
	\label{Eq(lambda(GC(X))<lambda(X))}
\end{equation}
If in particular $X$ is a bipartite $3$-valent graph, 
then for $i=1,2,\dots,\card{V(X)}$,
\begin{equation}
	\lambda_{\card{V(\GC_{k,l}(X))}-i+1}
	(\GC_{k,l}(X))
	\geq
	6-\frac{3k}{k^2+kl+l^2}\lambda_i(X).
	\label{Eq(lambda(GC(X))>6-lambda(X))}
\end{equation}
In the case that $l=0$, 
the last $\card{V(X)}$ eigenvalues 
of\/ $\GC_{k,0}(X)$ satisfy
\begin{equation}
	\lambda_{\card{V(\GC_{k,0}(X))}-i+1}
	(\GC_{k,0}(X))
	\geq
	\left\{
	\begin{aligned}
		& 3+\sqrt{5+4\cos\frac{2\pi}{k}},
		& & \text{if $X$ is $3$-valent},
		\\
		& 4+4\cos\frac{2\pi}{k},
		& & \text{if $X$ is $4$-valent and $k$ is even},
		\\
		& 4+4\cos\frac{\pi}{k},
		& & \text{if $X$ is $4$-valent and $k$ is odd},
	\end{aligned}
	\right.
	\label{Eq(lambda(GC(X))>6-e)}
\end{equation}
for $i=1,2,\dots,\card{V(X)}$. 
\end{Thm}
Moreover we have the following result. 
\begin{Thm}
\label{Thm(arbitrary_lambda)}
Let $X$ be a $3$-valent 
{\upshape(}resp.\ $4$-valent{\upshape)}
graph satisfying the same assumptions 
as in Theorem~\ref{Thm(comp_with_X)}. 
For any real number $\lambda\in [0,6]$ 
{\upshape(}resp.\ $\lambda\in [0,8]${\upshape)}, 
there exists a sequence $(\lambda_k)_k$ 
of eigenvalues of\/ $\GC_{k,0}(X)$ 
which converges to $\lambda$ as $k$ tends to infinity. 
\end{Thm}
As the following theorems show 
the Goldberg-Coxeter constructions have 
also steady eigenvalues. 
\begin{Thm}
\label{Thm(eigen24_3-vlnt)}
Let $X$ be a connected, finite and simple $3$-valent graph 
equipped with an orientation at each vertex, 
and $\GC_{2k,0}(X)$ be its Goldberg-Coxeter constructions 
for $k\in \mathbb{N}$. 
\begin{enumerate}
	\item 
	$\GC_{2k,0}(X)$ has eigenvalue $4$, 
	whose multiplicity is at least 
	$\lceil k/2\rceil$. 
	\item 
	$\GC_{2k,0}(X)$ has eigenvalue $2$, 
	whose multiplicity is at least 
	$\lfloor k/2\rfloor$. 
\end{enumerate}
\end{Thm}
In Theorem~\ref{Thm(eigen24_3-vlnt)}, 
$\lceil x\rceil$ (resp.\ $\lfloor x\rfloor$) 
denotes the smallest integer $\geq x$ 
(resp.\ the largest integer $\leq x$). 
\begin{Thm}
\label{Thm(eigen4_4-vlnt)}
Let $X$ be a connected, finite and simple 
$4$-valent graph 
equipped with an orientation at each vertex, 
and $\GC_{2k,0}(X)$ be its Goldberg-Coxeter constructions 
for $k\in \mathbb{N}$. 
Then, for $k\geq 2$, 
$\GC_{2k,0}$ has eigenvalue $4$, 
whose multiplicity is at least 
$\lceil (k-1)/2\rceil$. 
\end{Thm}
On the other hand, 
the multiplicities of eigenvalues $2$ and $4$ 
would depend on the graph structure of $X$ 
and the following is obtained.
\begin{Thm}
\label{Thm(3m-gons)}
Let $X$ be a connected, finite and simple 
$3$-valent graph which is embedded in a plane. 
Assume that the number of edges surrounding each face 
is divisible by $3$. 
Then the following hold. 
\begin{enumerate}
	\item \label{Item(3m-gons_1)}
	The multiplicity of eigenvalue $4$ 
	of\/ $\GC_{2,0}(X)$ is at least $3$. 
	\item \label{Item(3m-gons_2)}
	For any $k\in \mathbb{N}$, 
	both $\GC_{k,0}(X)$ and $\GC_{k,k}(X)$ 
	have eigenvalue $4$ {\upshape(}resp.\ $2${\upshape)}, 
	whose multiplicity 
	is at least $\lceil k/2\rceil$ 
	{\upshape(}resp.\ $\lfloor k/2\rfloor${\upshape)}. 
\end{enumerate}
\end{Thm}
The result (\ref{Item(3m-gons_1)}) of 
Theorem \ref{Thm(3m-gons)} is also obtained by 
observing that the $\GC_{2.0}(X)$ 
is a covering graph of the $K_4$ graph.
\par
Examples of numerical computations of 
multiplicities of eigenvalues $2$ and $4$ 
are shown in Tables \ref{Table(multiplicity:4)} 
and \ref{Table(multiplicity:2)}.
\begin{table}[htp]
  \centering
  \caption{The multiplicities of eigenvalue $4$ 
  for $\GC_{k,0}(X)$ ($k=1,2,\ldots,10$)}
  \begin{tabular}{l||c|c|c|c|c|c|c|c|c|c}
    X & $(1,0)$ & $(2,0)$ & $(3,0)$ & $(4,0)$ & $(5,0)$ & 
    $(6,0)$ & $(7,0)$ & $(8,0)$ & $(9,0)$ & $(10,0)$
    \\
    \hline \hline
    tetrahedron & $3$ & $6$ & $9$ & $12$ & $15$ & 
    $18$ & $21$ & $24$ & $27$ & $30$
    \\
    \hline
    cube & $3$ & $4$ & $3$ & $12$ & $3$ & 
    $20$ & $3$ & $28$ & $3$ & $36$
    \\
    \hline
    dodecahedron & $0$ & $6$ & $0$ & $18$ & $0$ & 
    $30$ & $0$ & $42$ & $0$ & $54$
    \\
    \hline
    octahedron & $3$ & $4$ & $3$ & $12$ & $3$ & 
    $20$ & $3$ & $28$ & $3$ & $36$
  \end{tabular}
  \label{Table(multiplicity:4)}
\end{table}
\begin{table}[htp]
  \centering
  \caption{The multiplicities of eigenvalue $2$ 
  for $\GC_{k,0}(X)$ ($k=1,2,\ldots,10$)}
  \begin{tabular}{l||c|c|c|c|c|c|c|c|c|c}
    X & $(1,0)$ & $(2,0)$ & $(3,0)$ & $(4,0)$ & $(5,0)$ & 
    $(6,0)$ & $(7,0)$ & $(8,0)$ & $(9,0)$ & $(10,0)$
    \\
    \hline \hline
    tetrahedron & $0$ & $3$ & $6$ & $9$ & $12$ & 
    $15$ & $18$ & $21$ & $24$ & $27$
    \\
    \hline
    cube & $3$ & $4$ & $3$ & $12$ & $3$ & 
    $20$ & $3$ & $28$ & $3$ & $36$
    \\
    \hline
    dodecahedron & $5$ & $6$ & $5$ & $18$ & $5$ & 
    $30$ & $5$ & $42$ & $5$ & $54$
    \\
    \hline
    octahedron & $0$ & $0$ & $1$ & $1$ & $0$ & 
    $1$ & $0$ & $1$ & $1$ & $0$
    \\
  \end{tabular}
  \label{Table(multiplicity:2)}
\end{table}
\par
Problems on eigenvalues of combinatorial Laplacian 
on regular graphs are extensively investigated. 
In particular, an explicit formula of 
a limit density of eigenvalue distributions of 
certain sequences of regular graphs was obtained 
in \cite{MR629617}, and 
its geometric proof using a trace formula is given 
in \cite{MR2218022} (see also \cite{MR1989434}). 
One of points in these works is that the sequence 
$\{X_{n}\}$ of $q$-regular graphs with number of vertices 
$\card{X_{n}}\to \infty$ as $n \to \infty$ 
is assumed to have large girths 
$g(X_{n}) \to \infty$ as $n \to \infty$. 
From this assumption, the graphs $X_{n}$ get similar, 
as $n\to \infty$, to a universal covering graph, 
namely a $q$-regular tree at least locally, and then 
a trace formula becomes able to apply. 
The girths of the Goldberg-Coxeter constructions 
$\{\GC_{k,l}(X)\}_{k,l}$ with an initial graph $X$ 
are uniformly bounded with respect to 
the parameters $k$ and $l$, 
and hence it would not be so straightforward 
to apply a trace formula to obtain a limit distribution 
of the eigenvalue distributions. 
\par
This paper is organized as follows. 
In Section~\ref{Sec(GC_construction)}, 
after giving the precise definition of 
the Goldberg-Coxeter constructions 
$\GC_{k,l}(X)$, 
we study their structure 
which is related with the spectral problems. 
In particular, variants of coloring problems 
necessary for our purposes are collected 
in Subsection \ref{Sec(conditions_on_3-vlnt_graphs)}. 
Some of them might be obtained from well-known results. 
For example, readers are referred to 
the interesting papers 
\cite{MR0335326, MR0335327, MR0335328, 
MR0357197, MR0454979, MR0498207} 
due to Fisk where one can find a lot of results 
on various kinds of coloring problems. 
However, we think that there are no statements 
which are precisely the same and 
the proofs do not involve so much. 
Thus we decided to put their proofs here 
for completeness. 
In Section~\ref{Sec(comp_of_ev)}, 
we obtain two kinds of comparisons of the eigenvalues, 
one is that between the eigenvalues of $X$ 
and those of $\GC_{k,l}(X)$, 
and the other is that between 
the eigenvalues of the $(k,0)$-cluster 
and those of $\GC_{k,0}(X)$. 
In Section~\ref{Sec(ev_of_clstr)}, 
all the eigenvalues of the $(k,0)$-cluster 
are found and the proofs of 
Theorems~\ref{Thm(o(k^2)_evs)}, 
\ref{Thm(comp_with_X)} and 
\ref{Thm(arbitrary_lambda)} complete. 
In Section~\ref{Sec(ev_2_4)}, 
we first present proofs of 
Theorem~\ref{Thm(eigen24_3-vlnt)} and 
\ref{Thm(eigen4_4-vlnt)}. 
At the end of this paper, 
we shall give a few criteria 
for a $3$-valent plane graph $X$ so that 
some $\GC_{k,0}(X)$'s have eigenvalues $2$ or $4$, 
which proves Theorem~\ref{Thm(3m-gons)}. 
%
\section{Goldberg-Coxeter constructions}
\label{Sec(GC_construction)}
This section studies the structure 
of Goldberg-Coxeter constructions, 
which shall be necessary in the subsequent sections. 
\par
The notion of Goldberg-Coxeter constructions is defined, 
due to Deza-Dutour~\cite{MR2429120,MR2035314}, 
for a plane graph. 
The definition can be extended for a nonplanar graph $X$; 
indeed, $X$ has only to be equipped with 
an ``orientation at each vertex'', and if, in particular, 
$X$ is ``appropriately'' embedded on an oriented surface, 
then the constructions can be done on the surface 
(see Proposition~\ref{Eq(edge_numbering)}). 
Let us give the precise definitions. 
To make description clear, 
we use the ring $\mathbb{Z}[\omega]$ of Eisenstein integers 
and the ring $\mathbb{Z}[i]$ of Gaussian integers, 
where $\omega=e^{\pi i/3}$ and $i=\sqrt{-1}$. 
$\mathbb{Z}[\omega]$ gives 
the triangular lattice on $\mathbb{C}$ having 
$0,1$ and $\omega$ as its fundamental triangle, 
while $\mathbb{Z}[i]$ gives the square lattice 
on $\mathbb{C}$ having 
$0,1,1+i$ and $i$ as its fundamental square. 
\begin{Def}[cf.\ Deza-Dutour~\cite{MR2429120,MR2035314}]
\label{Def(GC)}
Let $X$ be a connected, finite and simple 
$3$- or $4$-valent (abstract) graph 
equipped with an \emph{orientation at each vertex} 
in the sense that, for each $p\in V(X)$,  
the set of edges emanating from $p$ is ordered. 
For $(k,l)\in \mathbb{Z}^2$, $(k,l)\neq (0,0)$, 
the \emph{Goldberg-Coxeter construction} of $X$ 
with parameters $k$ and $l$ is 
defined through the following steps. 
\begin{enumerate}[label=(\roman*),ref=\roman*]
	\item 
          \label{enum:def:1}
	Let us first consider 
	the equilateral triangle 
	$\triangle=\triangle(0,z,\omega z)$ 
	in $\mathbb{Z}[\omega]$ 
	having the vertices $0,z=k+l\omega$ and $\omega z$ 
	(resp.\ the square 
	$\square=\square(0,z,(1+i)z,iz)$ 
	in $\mathbb{Z}[i]$ having the vertices 
	$0,z=k+li,(1+i)z$ and $iz$). 
	\item 
          \label{enum:def:2}
	Let us take all the small 
	triangles in $\mathbb{Z}[\omega]$ 
	(resp.\ squares in $\mathbb{Z}[i]$) 
	intersecting with $\triangle$ 
	(resp.\ $\square$) 
	in its interior and 
	join the barycenters of the neighboring small 
	triangles (resp.\ squares) 
	to obtain a graph, 
	which is, 
	as an associated (abstract) graph with $p$ 
	for each $p\in V(X)$, 
	denoted by $\overline{\triangle}(p)
	=\overline{\triangle}_{k,l}(p)$ 
	(resp.\ $\overline{\square}(p)
	=\overline{\square}_{k,l}(p)$). 
	Let us take a correspondence between 
	an edge emanating from $p$ and 
	an edge of $\triangle$ (resp.\ $\square$) 
	so that the given orientation at $p$ 
	coincides with the standard orientation of 
	$\triangle$ in $\mathbb{Z}[\omega]$ 
	(resp.\ $\square$ in $\mathbb{Z}[i]$). 
	Note that 
	$\overline{\triangle}(p)$ 
	(resp.\ $\overline{\square}(p)$) 
	has the $2\pi/3$-rotational symmetry 
	(resp.\ the $\pi/2$-rotational symmetry). 
	\item 
	For each $e\in E(X)$ with endpoints $p$ and $q$, 
	we can glue $\overline{\triangle}(p)$ and 
	$\overline{\triangle}(q)$ 
	(resp.\ $\overline{\square}(p)$ and 
	$\overline{\square}(q)$) 
	similarly as in the original definitions as follows:
	\begin{enumerate}[label=(iii-\arabic*)]
	\setlength{\leftskip}{15pt}
		\item 
		$\overline{\triangle}(p)$ 
		(resp.\ $\overline{\square}(p)$) is identified, 
		preserving the orientation, 
		with the graph on 
		$\triangle$ (resp.\ $\square$) 
		so that $e$ is corresponding to the edge 
		$\overline{z,\omega z}$ 
		(resp.\ $\overline{z,(1+i)z}$); 
		\item 
		$\overline{\triangle}(q)$ 
		(resp.\ $\overline{\square}(q)$) is identified, 
		preserving the orientation, 
		with the graph on 
		$\triangle(z,(1+\omega)z,\omega z)$ 
		(resp.\ $\square(z,2z,(2+i)z,iz)$) 
		so that $e$ is corresponding to the edge 
		$\overline{\omega z,z}$ 
		(resp.\ $\overline{(1+i)z,z}$);
		\item 
		then let us glue 
		$\overline{\triangle}(p)$ and 
		$\overline{\triangle}(q)$ 
		(resp.\ $\overline{\square}(p)$ and 
		$\overline{\square}(q)$) 
		by identifying all the vertices and edges 
		overlapping with each other.  
	\end{enumerate}
\end{enumerate}
The obtained (abstract) graph is again a 
$3$-valent (resp.\ $4$-valent) graph, 
which is denoted by $\GC_{k,l}(X)=\GC_z(X)$, 
where $z=k+l\omega$ (resp.\ $z=k+li$). 
\end{Def}
This definition, in general, may be not well-defined according to given orientation at each 
vertex of the graph. However, in case of plane graphs or graphs on an oriented surface, 
we have Proposition \ref{Prop(embedded_on_surface)} below.
\begin{figure}[htbp]
	\centering
	\begin{tabular}{ccc}
		\subfigure[a vertex $p$ and its adjancencies $q_1$, $q_2$, $q_3$]%
		{\includegraphics[height=3.7cm]%
		{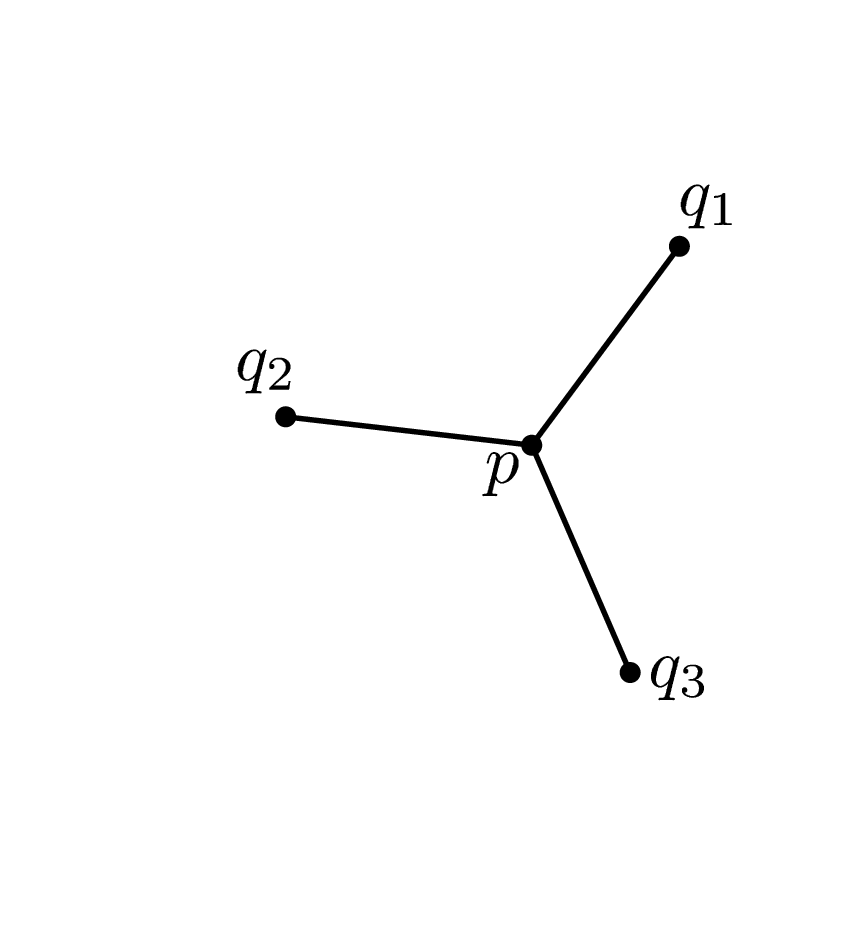}}
		&
		\subfigure[step (\ref{enum:def:1}) $\triangle(0,z,\omega z)$]%
		{\includegraphics[height=3.7cm]%
		{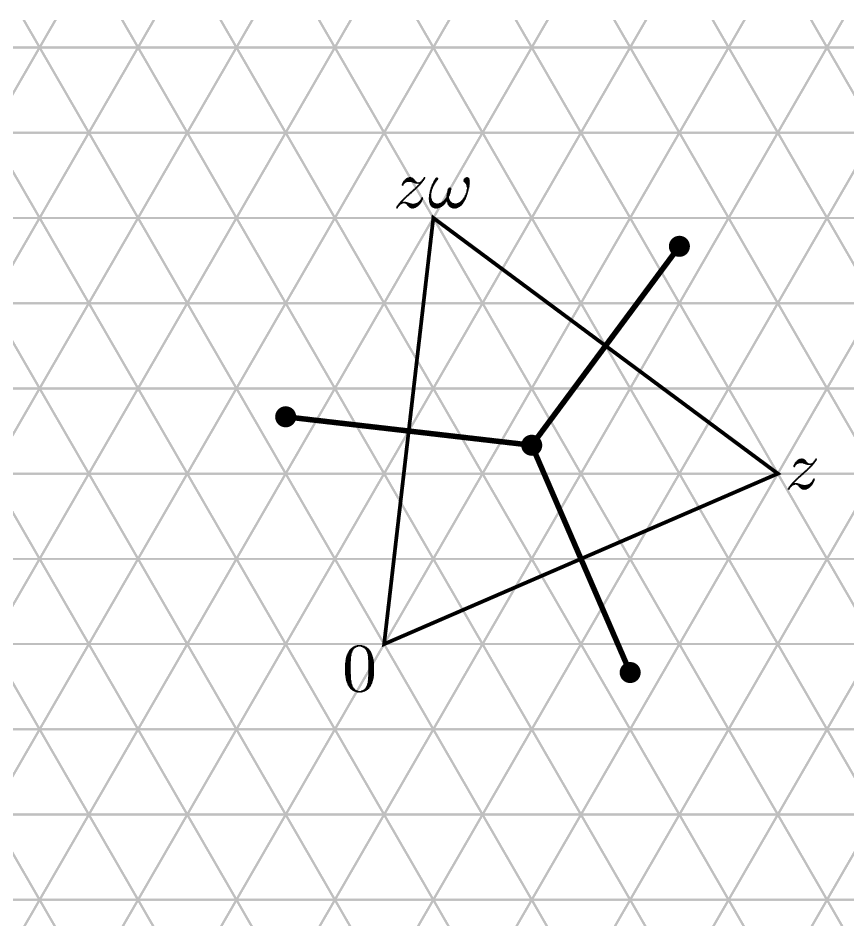}}
		&
		\subfigure[step (\ref{enum:def:2}) $\overline{\triangle}_{k,l}(p)$]%
		{\includegraphics[height=3.7cm]%
		{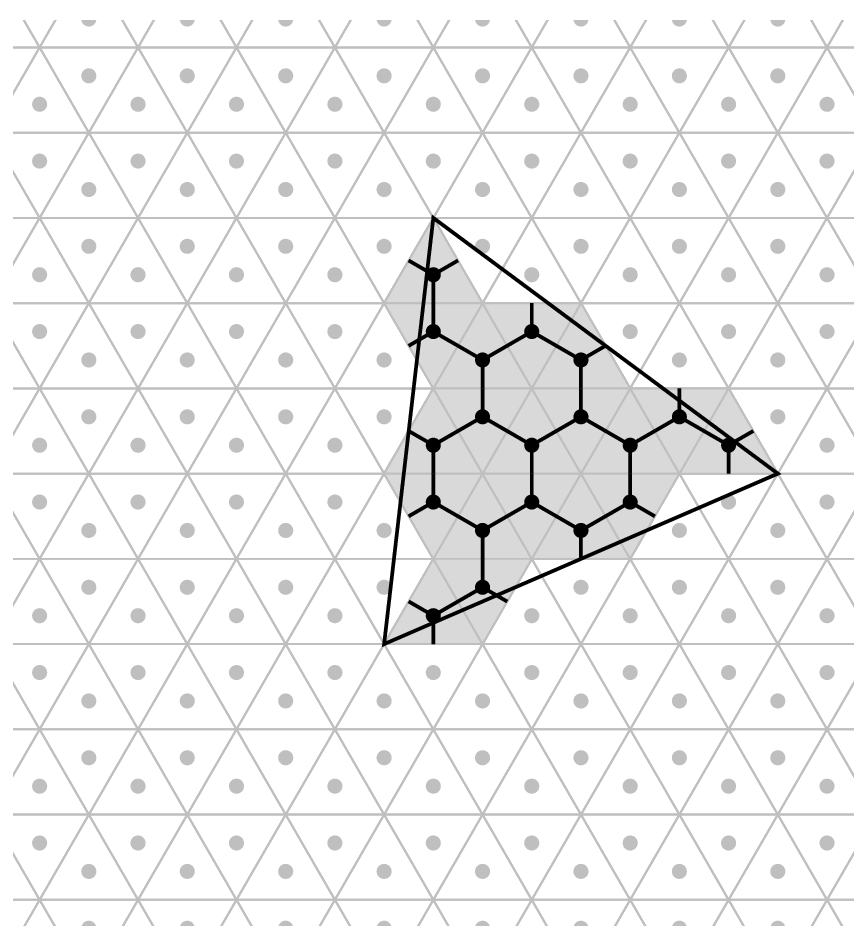}}
                  \\
		\subfigure[$V(\triangle_{k,l}(p))$]%
		{\includegraphics[height=3.7cm]%
		{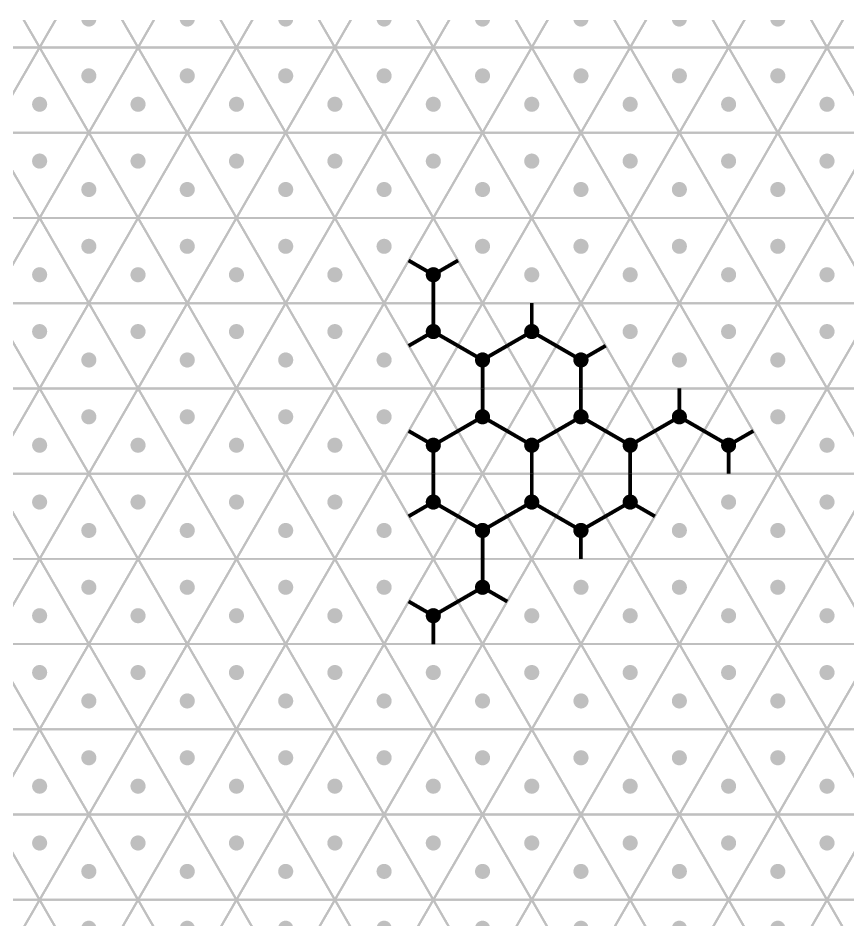}}
		&
		\subfigure[glue $\overline{\triangle}_{k,l}(p)$ and $\overline{\triangle}_{k,l}(q_i)$]%
		{\includegraphics[height=3.7cm]%
		{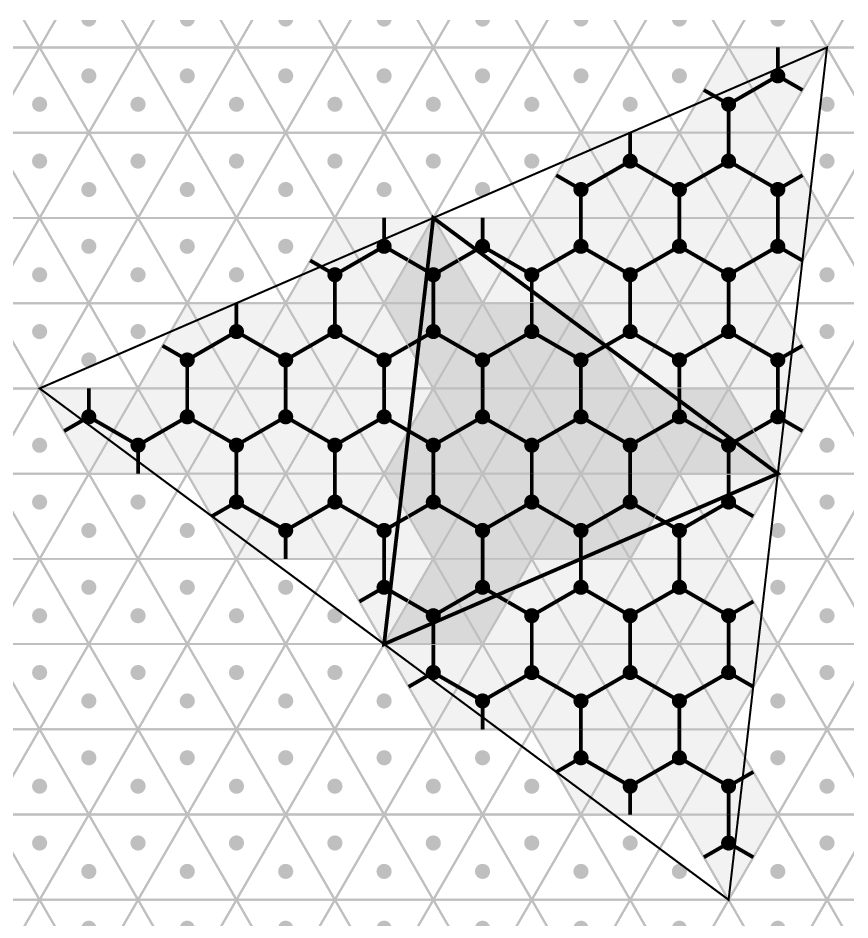}}
		&
		\subfigure[$V(\triangle_{k,l}(p)) \cup V(\triangle_{k,l}(q_i))$]%
		{\includegraphics[height=3.7cm]%
		{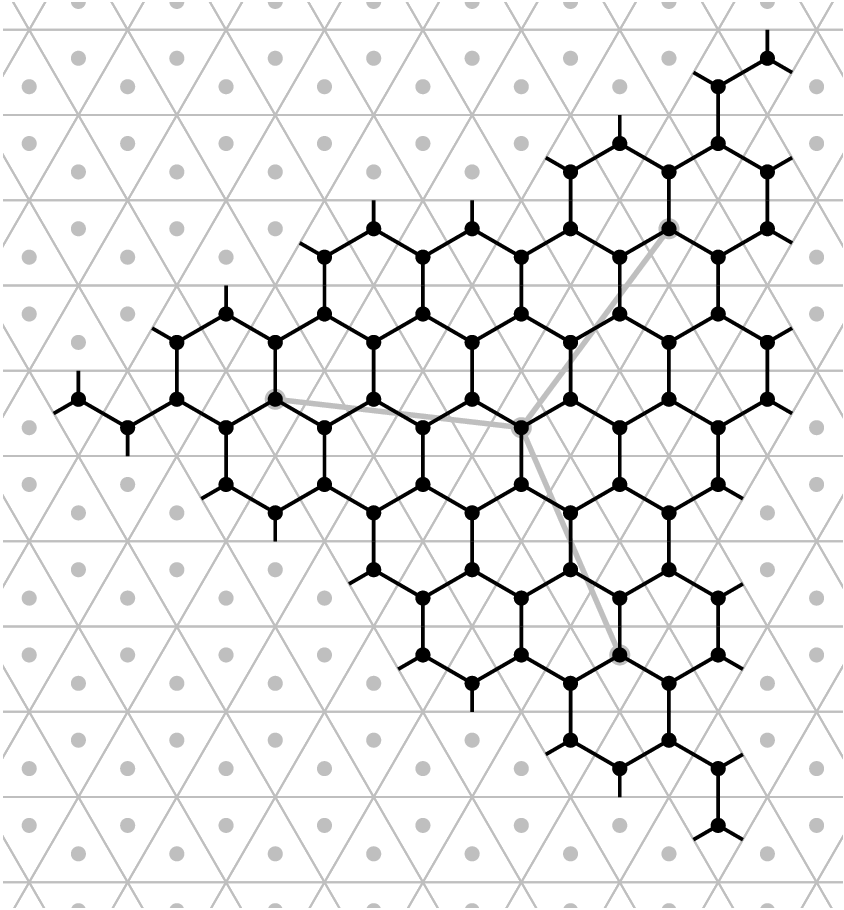}}
	\end{tabular}
	\caption{%
          In case of $(k, l) = (3, 2)$, the Goldberg-Coxeter construction for a 3-valent graph
          around a vertex $p$ of the graph.
          By this procedure, we construct $V(\triangle_{k,l}(p))$
          (see also Figure~\ref{Fig(3-vlnt_clstr)} and Section \ref{sec:cluster}).
	}
	\label{Fig(Definition)}
\end{figure}
\begin{Prop}
\label{Prop(embedded_on_surface)}
Let $X$ be a connected, finite and simple 
$3$- or $4$-valent graph 
which is embedded in an oriented surface $M$ in such a way 
that the closure of each face is simply connected. 
Then for $(k,l)\in \mathbb{Z}^2$, $(k,l)\neq (0,0)$, 
$\GC_{k,l}(X)$ is well-defined and 
is also embedded in $M$. 
\end{Prop}
\begin{proof}
The oriented tangent plane to $M$ at $p\in V(X)$ 
defines the orientation at $p$, 
and $\GC_{k,l}(X)$ is defined. 
The notion of faces is also well-defined. 
Since each face of $X$ is simply connected, 
we can take a dual graph $D_X$ of $X$ in $M$, 
all of whose faces are simply connected triangles 
(resp.\ rectangles) for the $3$-valent case 
(resp.\ $4$-valent case). 
The dividing step (ii) and 
the gluing step (iii) 
in Definition~\ref{Def(GC)} 
are well done in $M$ via 
respective appropriate local charts. 
\end{proof}
A Goldberg-Coxeter construction $\GC_{k,l}(X)$ for 
$3$-valent (resp.\ $4$-valent) graph $X$ 
inserts some hexagons (resp.\ squares), 
according to its parameter $k$ and $l$, 
between each pair of original faces of $X$. 
The most famous example is a fullerene $C_{60}$, 
called also a \emph{buckminsterfullerene} or 
a \emph{buckyball}, 
which is nothing but $\GC_{1,1}(\text{Dodecahedron})$. 
This construction owes its name to 
the pioneering work \cite{1937104} due to M.\ Goldberg, 
where 
a so-called Goldberg polyhedron 
(a convex polyhedron 
whose $1$-skeleton is a $3$-valent graph, consisting of 
hexagons and pentagons with rotational icosahedral symmetry
$3$-valent graph as its $1$-skeleton) 
is studied and is proved to be of the form 
$\GC_{k,l}(\text{Dodecahedron})$ for some $k$ and $l$. 
A Goldberg-Coxeter construction for 
$3$- or $4$-valent plane graphs occurs 
in many other context; 
see \cite{MR2429120} and the references therein. 
Several examples of Goldberg-Coxeter constructions 
for nonplanar $3$-valent 
(infinite or finite quotient) graphs, 
such as for carbon nanotubes and Mackay-like crystals, 
are provided in \cite{MR3724208}. 
\par
The following proposition summarizes few fundamental properties of 
Goldberg-Coxeter constructions. 
\begin{Prop}[Deza-Dutour~\cite{MR2429120,MR2035314}]
\label{Prop(properties_of_GC)}
Let $X=(V(X),E(X))$ be a $3$-valent 
{\upshape(}resp.\ $4$-valent{\upshape)} 
graph equipped with an orientation at each vertex. 
Then the following hold. 
\begin{enumerate}
	\item \label{Item(GC_zz'=GC_z(GC_z'))}
	If $X$ is embedded in an oriented surface 
	in such a way that 
	the closure of each face is simply connected, 
	and the orientation at each vertex coincides 
	with the one of the surface, then 
	$\GC_z(\GC_{z'}(X))=\GC_{zz'}(X)$, 
	for any $z,z'\in \mathbb{Z}[\omega]$ 
	{\upshape(}resp.\ $z,z'\in \mathbb{Z}[i]${\upshape)}. 
	\item \label{Item(isom_class_for_GC)}
	For any $(k,l)\in \mathbb{Z}^2$, $(k,l)\neq (0,0)$, 
	we have the following graph isomorphisms:
	\begin{gather*}
		\GC_{k,l}(X)
		\simeq \GC_{-l,k+l}(X)
		\simeq \GC_{-k-l,k}(X)
		\simeq \GC_{-k,-l}(X)
		\simeq \GC_{l,-k-l}(X)
		\simeq \GC_{k+l,-k}(X),
		\\
		\GC_{k,l}(X)
		\simeq \GC_{l,k}(X).
	\end{gather*}
	So $\{\GC_{k,l}(X)\mid k\geq l\geq 0,\ k\neq 0\}$ 
	gives a system of representatives of 
	graph isomorphism classes. 
	\item \label{Item(num_of_V(GC))}
	The number of vertices of\/ $\GC_z(X)$ is given as 
	$\card{V(\GC_z(X))}=\lvert z\rvert^2\card{V(X)}
	=(k^2+kl+l^2)\card{V(X)}$, 
	where $z=k+l\omega$ 
	{\upshape(}resp.\ 
	$\card{V(\GC_z(X))}=\lvert z\rvert^2\card{V(X)}
	=(k^2+l^2)\card{V(X)}$, 
	where $z=k+li${\upshape)}. 
\end{enumerate}
\end{Prop}
Deza-Dutour (\cite{MR2429120,MR2035314}) mentioned these properties 
only for plane graphs. 
However, these come from properties of triangular lattices and hence these also 
hold for graphs on oriented surfaces. 
In consideration of 
Proposition~\ref{Prop(properties_of_GC)} 
(\ref{Item(isom_class_for_GC)}), 
in the rest of this paper, 
we assume that $k$ is a positive integer 
and $l$ is a nonnegative integer satisfying 
$k\geq l\geq 0$ and $k\neq 0$. 
\subsection{Clusters for Goldberg-Coxeter constructions}
\label{Sec(cluster)}
A \emph{cluster} is the central notion in this paper. 
Its definitions shall be given below 
in two different cases: 
where $X$ is $3$-valent 
and where $X$ is $4$-valent. 
\subsubsection{The case where $X$ is $3$-valent} \label{sec:cluster}
For each $p\in V(X)$, 
let us construct a subgraph 
$\triangle_{k,l}(p)
=(V(\triangle_{k,l}(p)),E(\triangle_{k,l}(p)))$ 
of $\overline{\triangle}_{k,l}(p)\subseteq \GC_{k,l}(X)$, 
called the \emph{$(k,l)$-cluster}, 
so as to have $k^2+kl+l^2$ vertices 
and the $2\pi/3$-rotational symmetry 
of $\overline{\triangle}_{k,l}(p)$. 
For this, we just have to define 
$V(\triangle_{k,l}(p))$ 
by the set of vertices 
$x$ of $\overline{\triangle}_{k,l}(p)$ 
(considered as the graph on 
$\triangle\subseteq \mathbb{Z}[\omega]$) 
satisfying one of the following conditions:
\begin{enumerate}[label=(\roman*),ref=\roman*]
	\item \label{Item(3-clstr_intrr)}
	$x\in \overline{\triangle}_{k,l}(p)$ 
	corresponds to a triangle in $\mathbb{Z}[\omega]$ 
	whose barycenter lies in the interior of 
	$\triangle=\triangle(0,z,\omega z)$, 
	where $z=k+l\omega$;
	\item \label{Item(3-clstr_bdry)}
	$x\in \overline{\triangle}_{k,l}(p)$ 
	corresponds to an upward triangle in 
	$\mathbb{Z}[\omega]$ whose barycenter lies on 
	an edge of $\triangle$. 
\end{enumerate}
Here we mean an \emph{upward triangle} $\triangle(a,b)$ 
by the triangle in $\mathbb{Z}[\omega]$ with vertices 
$a+b\omega$, $a+1+b\omega$ and $a+(b+1)\omega$ 
for $a,b\in \mathbb{Z}$ 
(see Figure~\ref{Fig(3-vlnt_clstr)}). 
We also denote by $\dtriangle(a,b)$, 
called \emph{downward triangle}, the triangle with vertices 
$a+b\omega$, $a+(b+1)\omega$ and $a-1+(b+1)\omega$. 
\par
In the case that $l=0$, 
$\triangle_{k,0}(p)$ is nothing but 
$\overline{\triangle}_{k,0}(p)$ itself, 
has $k^2$ vertices and 
has the dihedral symmetry $D_3$ (of order $6$) 
(see Figure~\ref{Fig(3-vlnt_clstr)}). 
\par
In the case that $k=l>0$, 
it is easily seen that there are $3(k^2-k)$ 
vertices satisfying (\ref{Item(3-clstr_intrr)}) and 
$3k$ vertices satisfying (\ref{Item(3-clstr_bdry)}). 
The obtained subgraph 
$\triangle_{k,k}(p)$ 
has $3k^2$ vertices and 
has the $2\pi/3$-rotational symmetry 
because upward triangles are mapped to upward triangles 
by the rotation 
(see Figure~\ref{Fig(3-vlnt_clstr)}). 
\par
The following lemma makes clear 
the cases where there is a barycenter lying on an edge of 
$\triangle$ among the remaining cases. 
\begin{Lem}
\label{Lem(through_barycenter)}
Let $k>l>0$, $m:=\gcd(k,l)$, $k_1:=k/m$ and $l_1:=l/m$. 
An edge of the triangle 
$\triangle=\triangle(0,z,\omega z)$, where $z=k+l\omega$ 
in $\mathbb{Z}[\omega]$ 
passes through a barycenter of 
a small triangle in $\mathbb{Z}[\omega]$ 
if and only if 
\begin{equation}
	k_1\not\equiv 0\pmod 3,
	\qquad
	k_1\equiv l_1\pmod 3.
\label{Eq(k-l=0_mod3)}
\end{equation}
Moreover, in the case above, each edge of\/ $\triangle$ 
passes through exactly $2m=2\gcd(k,l)$ barycenters. 
Among these $2m$ vertices, 
exactly $m$ vertices corresponding to upward triangles 
have just two adjacent triangles 
with barycenters lying in $\triangle$. 
The combined $3m$ vertices 
on the three edges of\/ $\triangle$ 
are located in symmetric position with 
the rotation by $2\pi/3$ of\/ $\triangle$. 
\end{Lem}
Lemma~\ref{Lem(through_barycenter)} shows that 
the subgraph $\triangle_{k,l}(p)$ 
has $(k-l)^2+3kl=k^2+kl+l^2$ 
vertices and also has the $2\pi/3$-rotational symmetry 
in the remaining case that $k>l>0$. 
\begin{figure}[htbp]
	\centering
	\begin{tabular}{ccc}
		\subfigure[$(k,l)=(5,0)$]%
		{\includegraphics[height=3.7cm]%
		{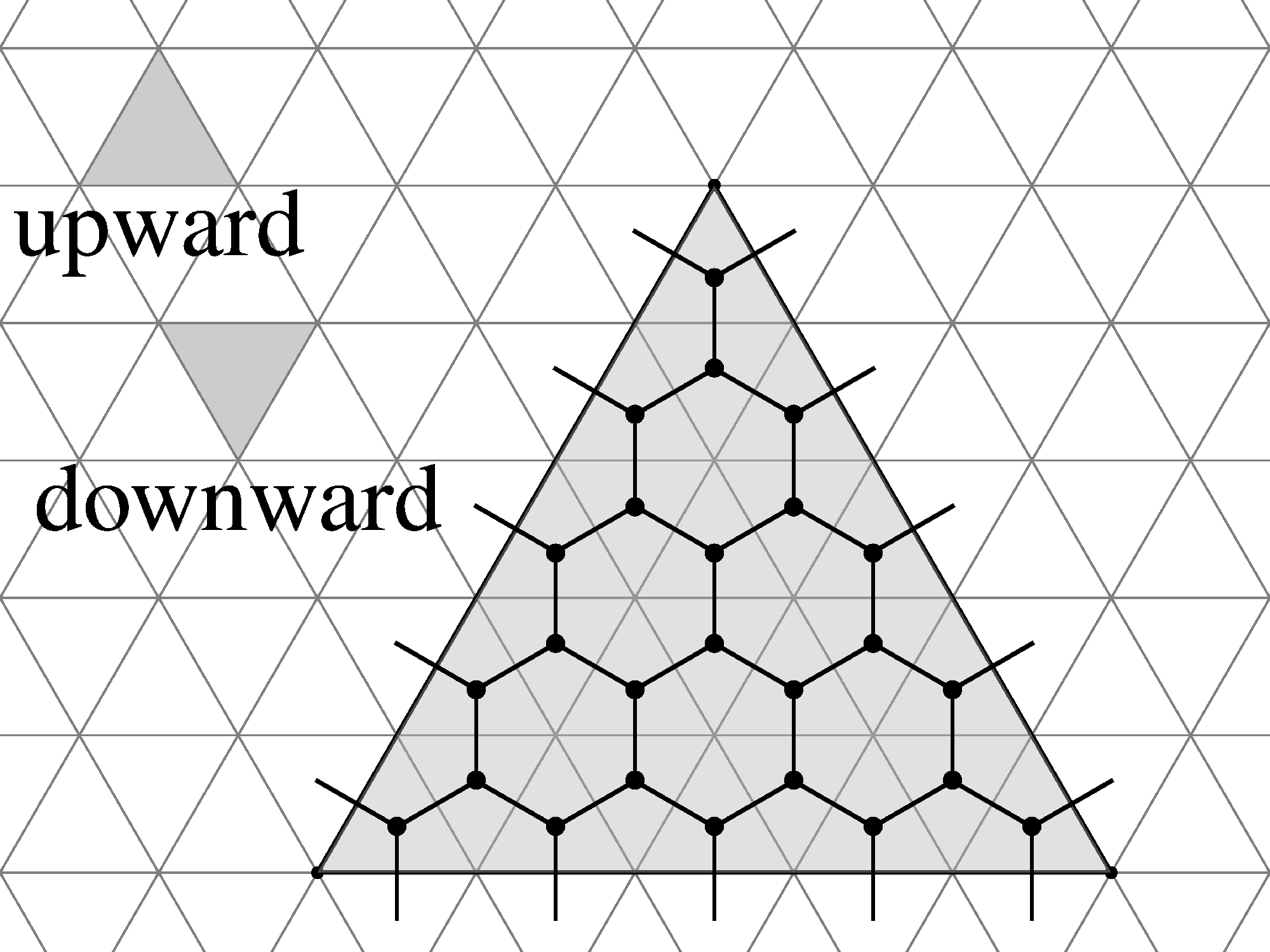}}
		&
		\subfigure[$(k,l)=(3,3)$]%
		{\includegraphics[height=3.7cm]%
		{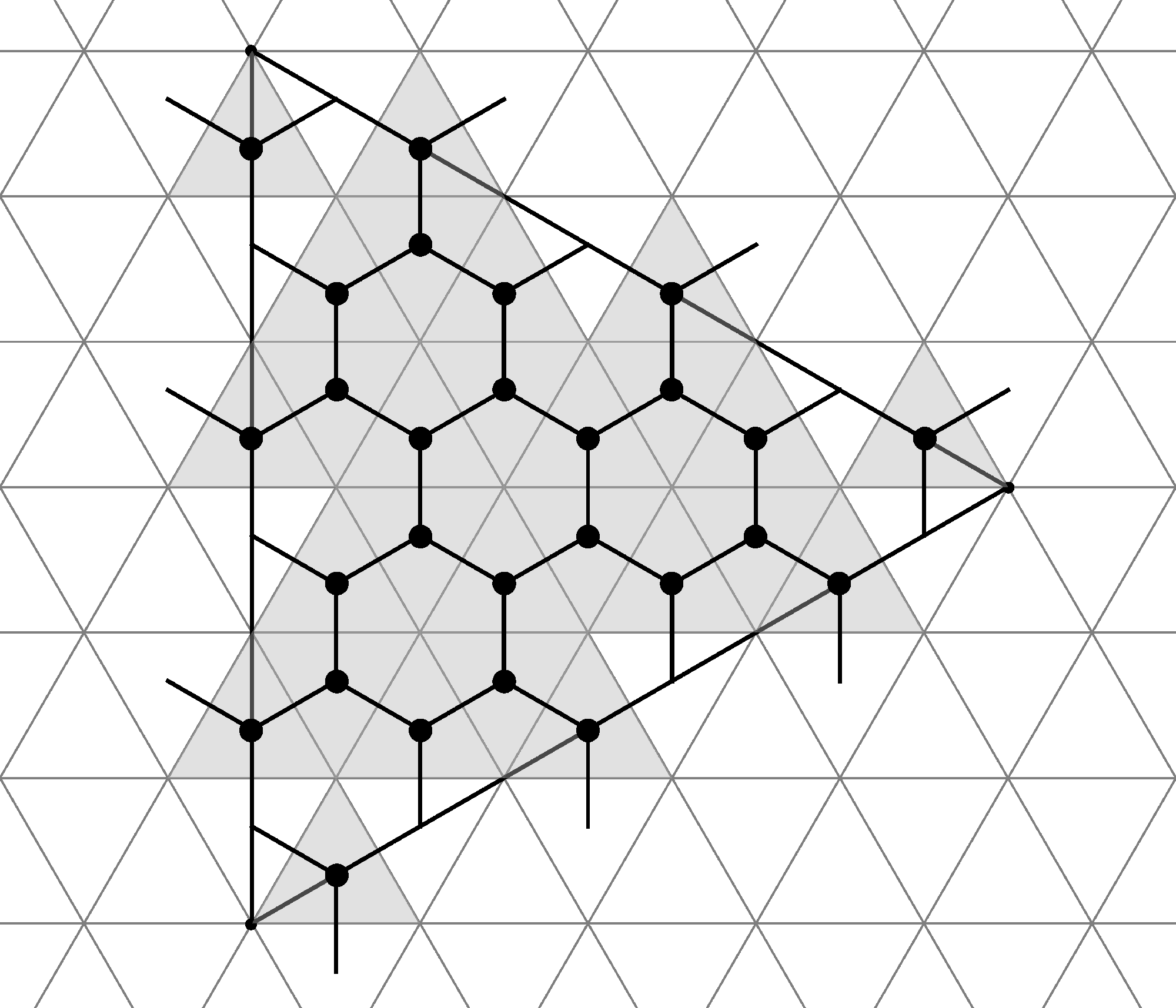}}
		&
		\subfigure[$(k,l)=(4,1)$]%
		{\includegraphics[height=3.7cm]%
		{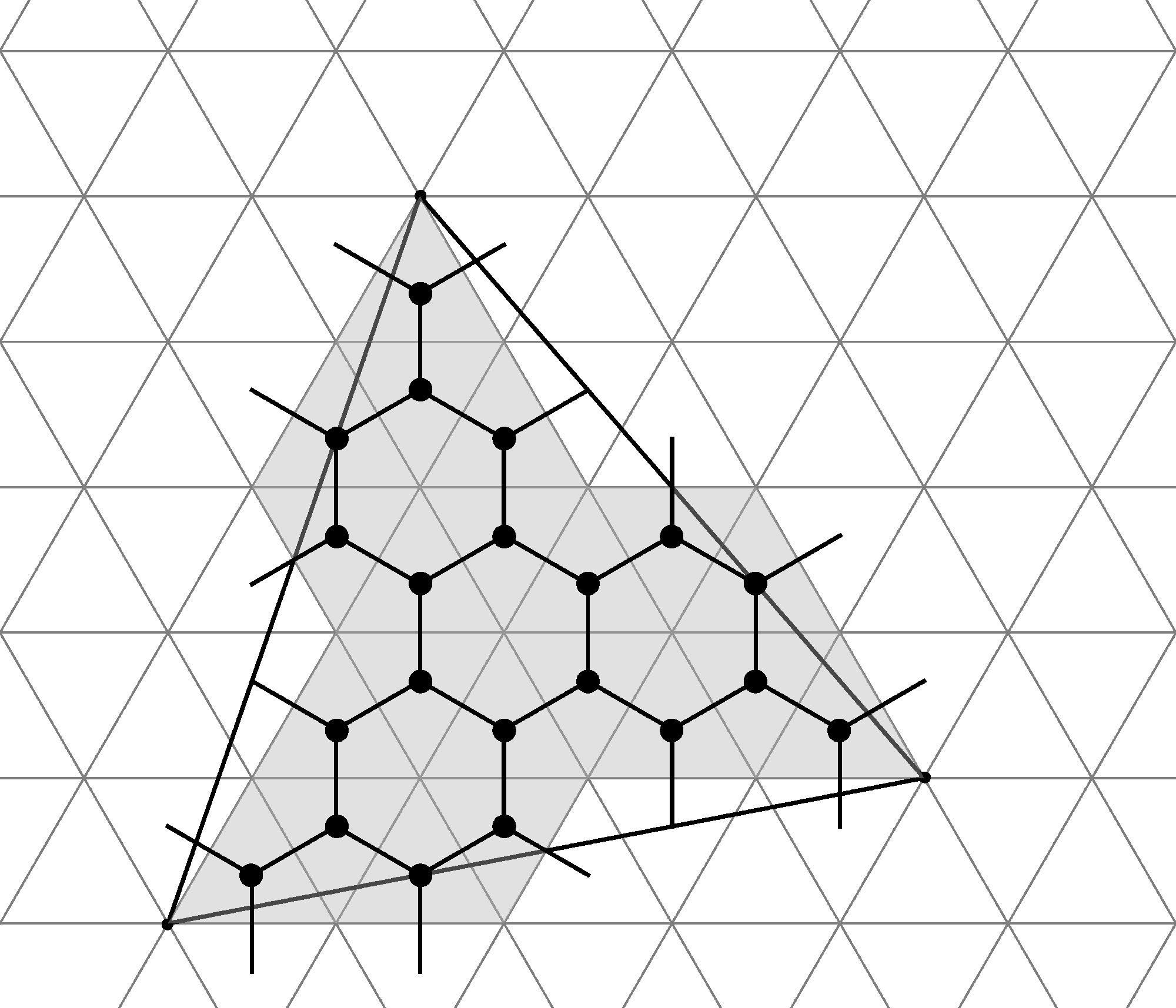}}
	\end{tabular}
	\caption{%
		$3$-valent $(k,l)$-clusters; 
		$V(\triangle_{k,l}(p))$ 
		consists of the barycenters of 
		the gray triangles. 
	}
	\label{Fig(3-vlnt_clstr)}
\end{figure}
\par
Here we can prove the following proposition, 
which guarantees that 
the bipartiteness is kept 
after a Goldberg-Coxeter construction. 
\begin{Prop}
\label{Prop(GC_bipartite)}
Let $X$ be a $3$-valent bipartite graph 
equipped with an orientation at each vertex. 
Then for any $(k,l)\in \mathbb{Z}^2$, $(k,l)\neq (0,0)$, 
$\GC_{k,l}(X)$ is also bipartite. 
So the spectrum of\/ $\GC_{k,l}(X)$ is symmetric 
with respect to $3$. 
\end{Prop}
\begin{proof}
Let a bipartition of $X$ be given 
and either black or white be assigned 
to each vertex $p\in V(X)$. 
Each vertex $x$ of each $\triangle_{k,l}(p)$ 
can be colored according to a rule that 
if $p$ is white, then
\begin{itemize}
	\item paint $x$ black, 
	provided the triangle 
	in $\mathbb{Z}[\omega]$ 
	corresponding to $x$ is upward; 
	\item paint $x$ white, 
	provided the triangle 
	in $\mathbb{Z}[\omega]$ 
	corresponding to $x$ is downward; 
\end{itemize}
and if $p$ is black, then exchange 
black and white above. 
A white vertex is adjacent 
only to black vertices in $X$, 
and two adjacent clusters 
$\triangle_{k,l}(p)$ and $\triangle_{k,l}(q)$ 
are positioned, in $\mathbb{Z}[\omega]$, at $\pi$-rotation 
around the midpoint of an edge of $\triangle$, 
which switches upward and downward triangles. 
So, the rule above gives a bipartition of $\GC_{k,l}(X)$. 
\end{proof}
\subsubsection{The case where $X$ is $4$-valent}
Similarly as in the $3$-valent case, 
we construct for each $p\in V(X)$ 
an appropriate subgraph 
$\square_{k,l}(p)=(V(\square_{k,l}(p)),E(\square_{k,l}(p)))$ 
of $\overline{\square}_{k,l}(p)$, 
still called the \emph{$(k,l)$-cluster}, 
so as to have $k^2+l^2$ vertices. 
To this end, we need to clarify the cases 
where a barycenter of a small square in $\mathbb{Z}[i]$ 
lies on an edge of $\square=\square(0,z,(1+i)z,iz)$, 
where $z=k+li$. 
\par
For $a,b\in \mathbb{Z}$, 
we denote by $\square(a,b)$ 
the small square in $\mathbb{Z}[i]$ 
with vertices $a+bi,(a+1)+bi,(a+1)+(b+1)i,a+(b+1)i$, 
whose barycenter is given as $a+1/2+(b+1/2)i$. 
\begin{Lem}
Let $k\geq l\geq 0$, $k\neq 0$, 
$m:=\gcd(k,l)$, $k_1:=k/m$ and $l_1:=l/m$. 
An edge of the square 
$\square=\square(0,z,(1+i)z,iz)$, 
where $z=k+li$ in $\mathbb{Z}[i]$ 
passes through a barycenter of 
a small square in $\mathbb{Z}[i]$ 
if and only if 
\begin{equation}
	k_1\not\equiv 0\pmod 2,
	\qquad
	k_1\equiv l_1\pmod 2
\label{Eq(k-l=0_mod2)}
\end{equation}
Moreover, if this is the case, 
each edge of\/ $\square$ passes through 
exactly $m$ barycenters. 
\end{Lem}
Unlike the $3$-valent case, 
we cannot choose a cluster $\square_{k,l}(p)$ 
with $k^2+l^2$ vertices 
to have the $\pi/2$-rotational symmetry 
in the case where $k_1\not\equiv 0\pmod 2$, 
$k_1\equiv l_1\pmod 2$ and $m\not\equiv 0\pmod 2$
because 
no vertex of $\overline{\square}_{k,l}(p)$ is positioned 
at the barycenter of $\square$ and 
$k^2+l^2=m^2((k_1-l_1)^2+2k_1l_1)$ 
is not divided by $4$. 
Even in such cases, $\square_{k,l}(p)$ only has 
to have the same number of outward edges 
among the four directions to every adjacent cluster. 
\begin{Lem}[cf.\ {\cite[Corollary IV.6]{MR1055084}}]
\label{Lem(euler_circuit)}
Let $X$ be a $4$-valent graph 
equipped with an orientation at each vertex. 
Then there exists an Euler circuit $\varepsilon$ of $X$ 
which turns either left or right 
at every vertex of $X$. 
\end{Lem}
We note that the Euler circuit mentioned 
in Lemma \ref{Lem(euler_circuit)} 
is an example of $A$-trails introduced 
in \cite{MR1055084}.
\begin{proof}
As is well-known, any $4$-valent graph $X$ 
has an Euler circuit, 
which is by definition a closed path in $X$ 
which visits every edge exactly once. 
Let us take an Euler circuit $\varepsilon$ of $X$ 
and suppose that $\varepsilon$ goes straight ahead 
at a vertex $p\in V(X)$. 
The circuit $\varepsilon$ comes back to $p$ 
again from one of the other directions 
after it straight ahead at $p$ 
(because $X$ is $4$-valent). 
By following the interval in opposite directions, 
the obtained circuit goes straight ahead 
one time fewer than $\varepsilon$. 
This proves Lemma~\ref{Lem(euler_circuit)}. 
\end{proof}
The Euler circuit $\varepsilon$ 
obtained in Lemma~\ref{Lem(euler_circuit)} 
assigns a direction to each edge of $X$ 
such that the direction alternates 
between inward and outward at each vertex of $X$. 
\par
Now we can clearly define 
$V(\square_{k,l}(p))$ 
by the set of vertices $x$ of 
$\overline{\square}_{k,l}(p)$ 
satisfying one of the following conditions:
\begin{enumerate}[label=(\roman*)]
	\item \label{Item(4-clstr_intrr)}
	$x$ corresponds to a square in $\mathbb{Z}[i]$ 
	whose barycenter lies in the interior of 
	$\square$;
	\item \label{Item(4-clstr_bdry)}
	$x$ corresponds to a barycenter lying on 
	the two edges of $\square$ with opposite sides 
	which correspond to the outward edges of $X$ 
	with respect to the Euler circuit 
	$\varepsilon$ in Lemma~\ref{Lem(euler_circuit)}. 
\end{enumerate}
\begin{figure}[htbp]
	\centering
	\begin{tabular}{ccc}
		\subfigure[$(k,l)=(5,0)$]%
		{\includegraphics[width=.3\textwidth]%
		{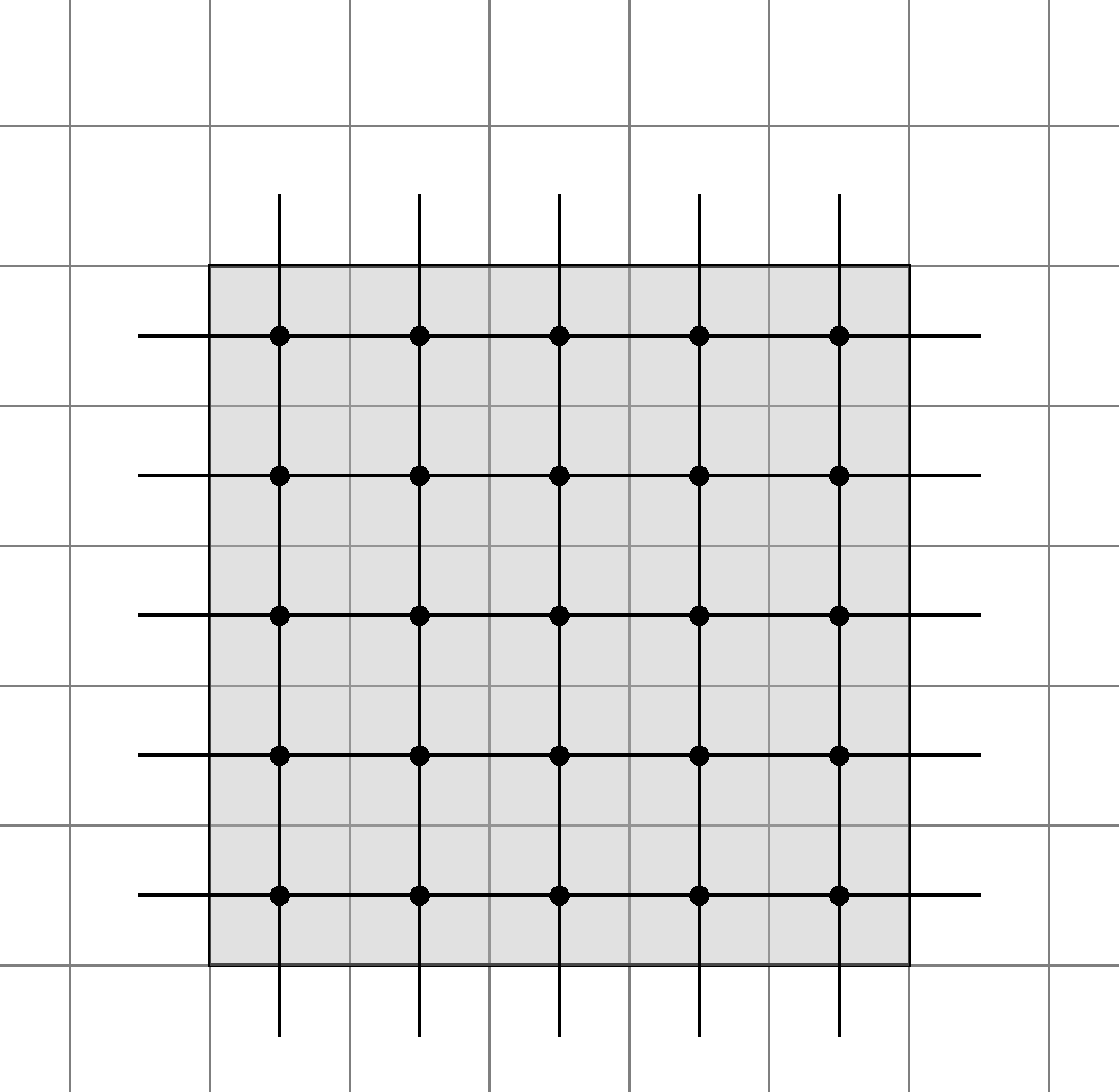}}
		&
		\subfigure[$(k,l)=(3,3)$]%
		{\includegraphics[width=.3\textwidth]%
		{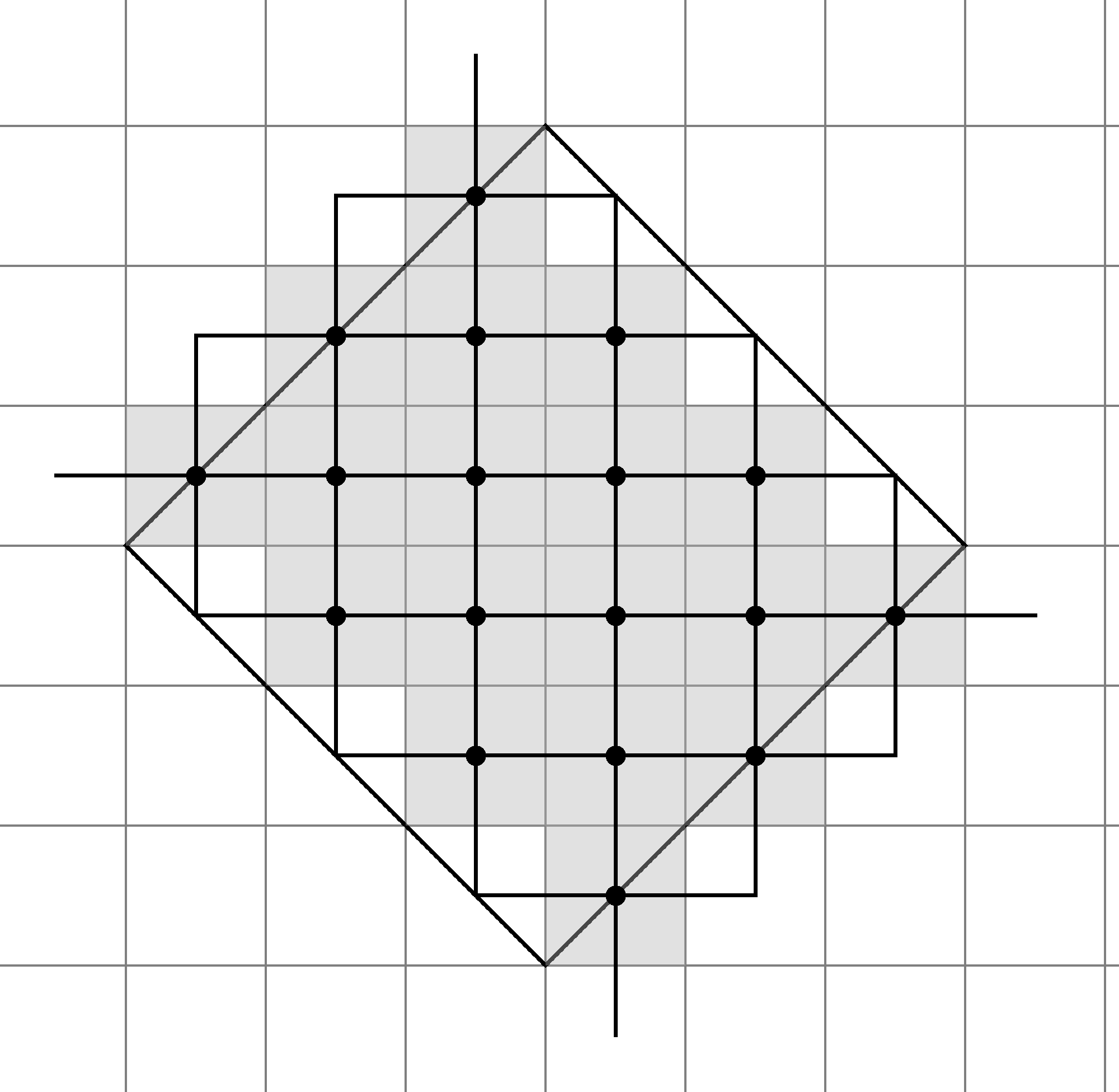}}
		&
		\subfigure[$(k,l)=(5,1)$]%
		{\includegraphics[width=.3\textwidth]%
		{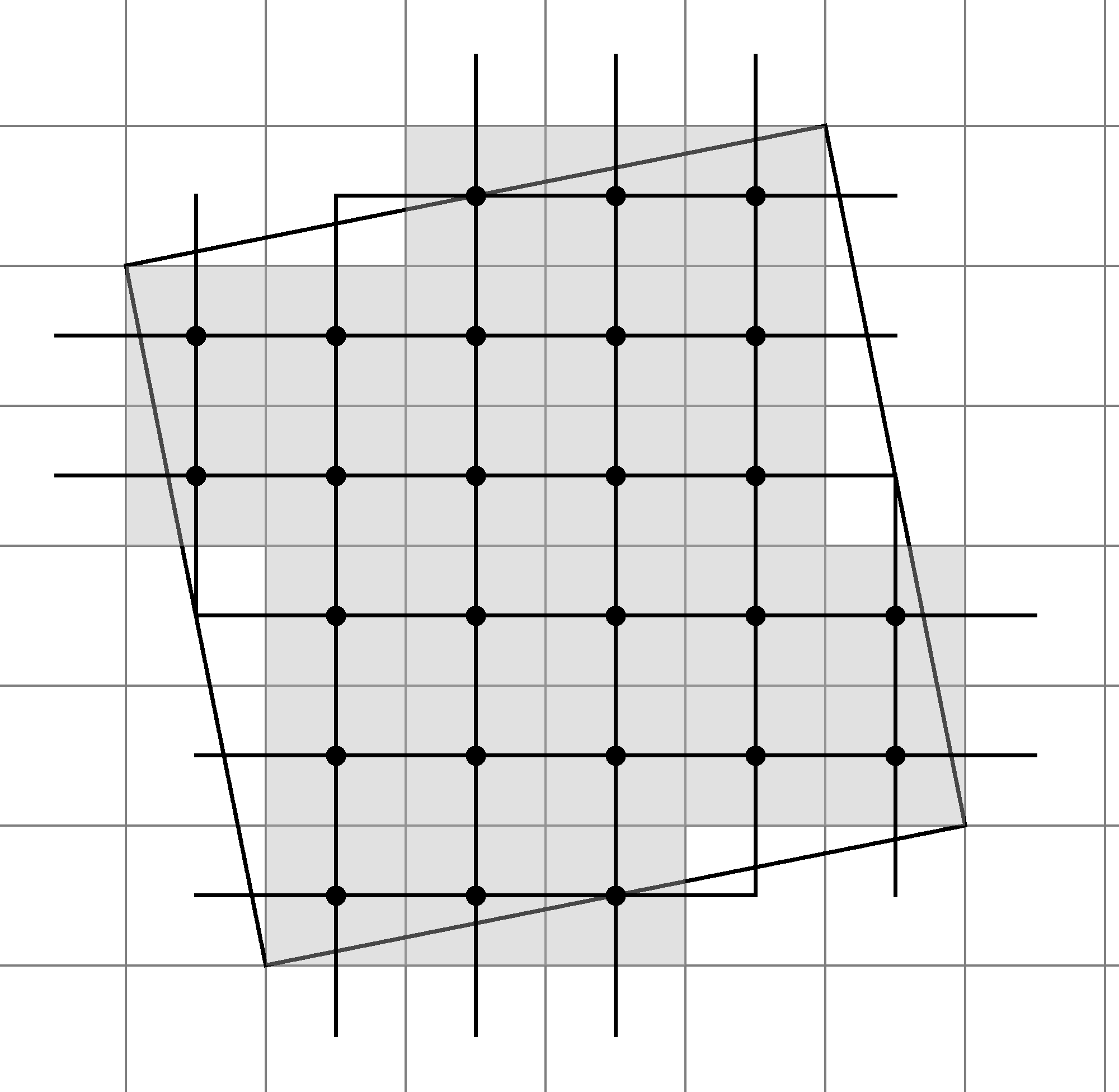}}
	\end{tabular}
	\caption{%
		$4$-valent $(k,l)$-clusters; 
		$V(\square_{k,l}(p))$ consists of the barycenters of 
		the gray squares. 
	}
\end{figure}
\subsection{Specific conditions on $3$-valent plane graphs}
\label{Sec(conditions_on_3-vlnt_graphs)}
In this subsection 
a graph $X$ is assumed to be 
$3$-valent and \emph{embedded in a plane}, 
and we study some structure of 
Goldberg-Coxeter constructions of $X$. 
Let us consider the 
following four conditions on $X$, 
which is, as shall be seen in Section~\ref{Sec(3m-gons)}, 
related to the multiplicities of certain eigenvalues 
of Goldberg-Coxeter constructions. 
\begin{enumerate}[itemsep=5pt,topsep=5pt]
	\item[(F)] 
	The number of edges surrounding each face is 
	divisible by $3$. 
	\item[(CN)] 
	For each vertex $p\in V(X)$, 
	the numbers $1,2$ and $3$ are assigned in this order, 
	with respect to the positive orientation, 
	to the three edges of $X$ with $p$ 
	as the common endpoint. 
	\item[(N)] 
	There exists a vertex numbering 
	$V(\GC_{2,0}(X))\rightarrow \{1,2,3\}$ 
	with the following properties:
	\begin{enumerate}[topsep=3pt]
	\setlength{\leftskip}{18pt}
		\item[(N-i)] 
		The number $0$ is assigned 
		to the center of each $V(\triangle_{2,0}(p))$
		($p\in X$); 
		\item[(N-ii)] 
		the number assigned to $x\in V(\GC_{2,0}(X))$ 
		is different from those of 
		the adjacent vertices 
		in $\GC_{2,0}(X)$ of $x$. 
	\end{enumerate}
	\item[(C)] 
	$V(X)$ can be colored by two colors, 
	say black and white, 
	with the following properties:
	\begin{enumerate}[topsep=3pt]
	\setlength{\leftskip}{18pt}
		\item[(C-i)] 
		A black vertex is adjacent to three white vertices; 
		\item[(C-ii)] 
		a white vertex is adjacent to 
		exactly one black vertex, 
		so the other two adjacent vertices are white; 
		\item[(C-iii)] 
		for any pair of black vertices 
		$x,y\in V(X)$ which are three vertices 
		away from each other, 
		there is a path from $x$ to $y$ 
		either turning left twice or 
		turning right twice. 
	\end{enumerate}
\end{enumerate}
\begin{Rem}
\label{Rem(condition)}
We remark that
\begin{enumerate}
	\item \label{Item(condition-1)}
	the condition (N) determines 
	a special $3$-edge-coloring of $\GC_{2,0}(X)$, 
	but a graph with $3$-edge-coloring 
	does not necessarily satisfy (N),  
	\item \label{Item(condition-2)}
	a $3$-valent plane graph 
	which satisfies the condition (F) 
	is known to be a covering graph of the $K_4$ graph. 
	(This fact is proved also from 
	Proposition~\ref{Prop(coherent_numbering)} below.\ )
\end{enumerate}
\end{Rem}
The \emph{coherent edge numbering} (CN) 
implies the condition (N); indeed, 
let $p\in V(X)$ and let $e_1,e_2$ and $e_3$ be three edges 
of $X$ emanating from $p$. 
We assign $0$ to $p$ regarded as 
a vertex of $\GC_{2,0}(X)$, and, for $i=1,2$ and $3$, 
assign $i$ to the vertex of $\GC_{2,0}(X)$ positioned 
at the ``opposite-side'' to $e_i$. 
The resulting numbering of vertices of $\GC_{2,0}(X)$ 
satisfies (N-i) and (N-ii) 
(see Figure~\ref{Fig(coherent_number)}). 
Moreover, as is easily proved, (N) 
implies the condition (F). 
So the following proposition shows that 
(F), (CN) and (N) are mutually equivalent. 
\begin{figure}[htbp]
	\centering
	\includegraphics[height=4cm]%
	{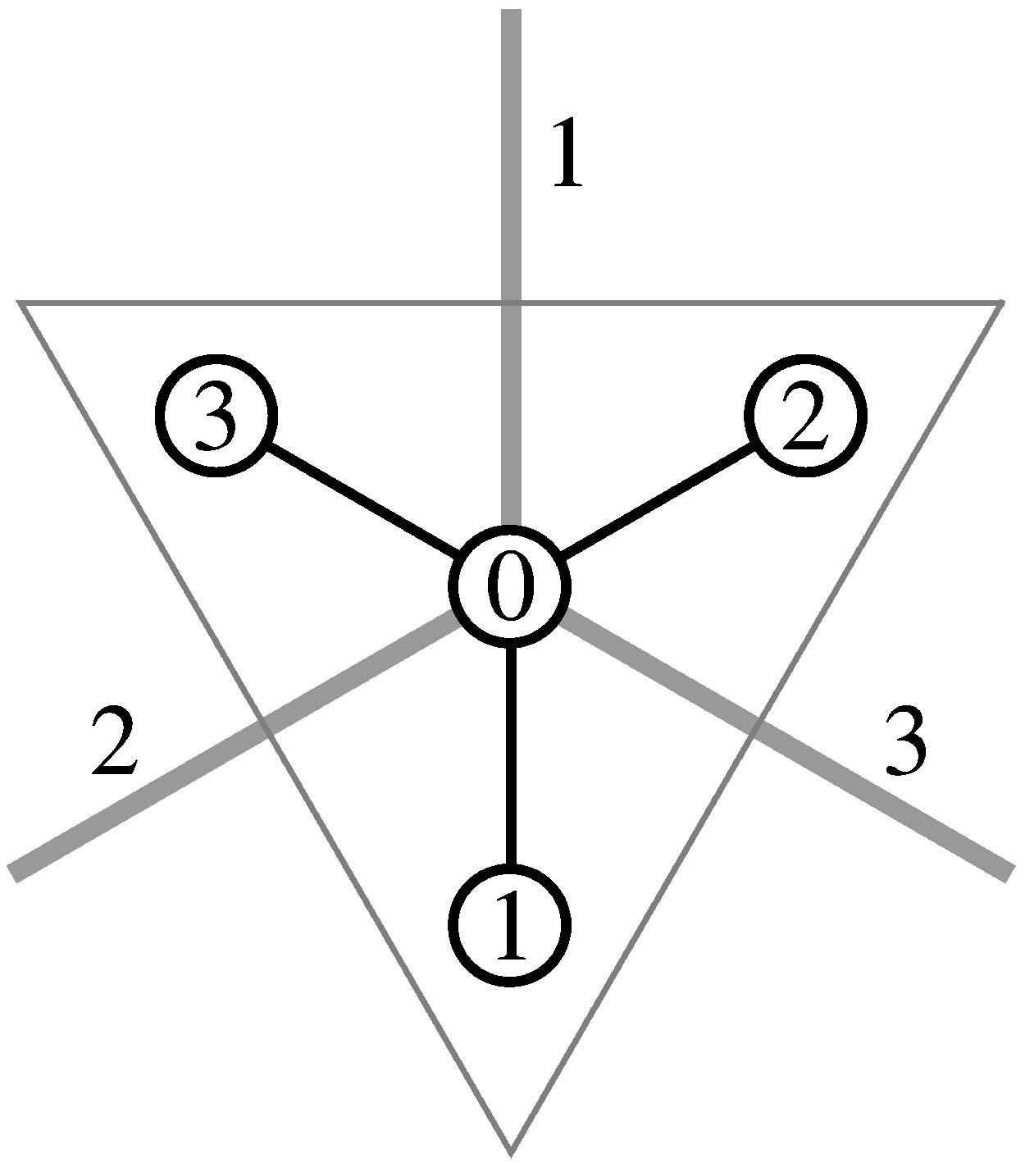}
	\caption{The gray (resp.\ black) segments 
	represent edges of $X$ (resp.\ $\GC_{2,0}(X)$). }
	\label{Fig(coherent_number)}
\end{figure}
\begin{Prop}\label{Prop(coherent_numbering)}
Let $X$ be a $3$-valent plane graph 
satisfying {\upshape(F)}. 
Then $X$ has a coherent edge numbering 
$E(X)\rightarrow \{1,2,3\}$ satisfying 
{\upshape(CN)}. 
\end{Prop}
\begin{proof}
For a sequence 
$e_1,e_2,\dots,e_k\in E(X)$ 
of adjacent edges in $X$ 
(namely $e_i\neq e_{i+1}$ and they have a common endpoint), 
we denote by $(e_1,e_2,\dots,e_k)$ 
the path along the edges 
starting from $m(e_1)$, the midpoint of $e_1$, 
and ending with $m(e_k)$. 
To get a desired numbering 
$\nu\colon E(X)\rightarrow \{1,2,3\}$, 
we fix an edge $e_0\in E(X)$, 
and assign $3$ to $e_0$. 
For adjacent edges $e,e'\in E(X)$ 
with the common endpoint $p$, 
let us define $\tau(e,e')$ as
\begin{equation}
	\tau(e,e')
	:=
	\left\{
	\begin{aligned}
	&+1,& & 
	\text{if the path $(e,e')$ turns right at $p$},
	\\
	&-1,& & 
	\text{if the path $(e,e')$ turns left at $p$},
	\end{aligned}
	\right.
\label{Eq(tau(e,e'))}
\end{equation}
Note that $\tau(e,e')=-\tau(e',e)$. 
We then define 
$\nu\colon E(X)\rightarrow \{1,2,3\}$ 
for $e\in E(X)$ as
\begin{equation}
	\nu(e)
	:=\sum_{i=1}^n\tau(e_{i-1},e_i)\pmod 3
\label{Eq(edge_numbering)}
\end{equation}
by choosing a path 
$\gamma=(e_0,e_1,\dots,e_{n-1},e_n=e)$ 
from $e_0$ to $e$. 
What we have to prove is that 
$\nu(e)$ is independent of the choice of $\gamma$. 
To this end, 
let $\mathcal{P}(X,e_0)$ be the set of sequences of 
adjacent edges in $X$ beginning with $e_0$ 
and let 
$\varphi\colon \mathcal{P}(X,e_0)\rightarrow \{1,2,3\}$ 
be a map defined as
\[
	\varphi(\gamma)
	:=\sum_{i=1}^n\tau(e_{i-1},e_i)\pmod 3,
	\quad
	\text{for 
	$\gamma=(e_0,e_1,\dots,e_n)\in \mathcal{P}(X,e_0)$}.
\]
Note that for any $\gamma=(e_0,e_1,\dots,e_n),
\gamma'=(e_0,e_1',\dots,e_{m-1}',e_n)
\in \mathcal{P}(X,e_0)$,
the joined (closed) path 
$\gamma'^{-1}\cdot \gamma
=(e_0,e_1,\dots,e_n,e_{m-1}',\dots,e_1',e_0)$ 
satisfies
\[
	\varphi(\gamma'^{-1}\cdot \gamma)
	\equiv\varphi(\gamma)-\varphi(\gamma')\pmod 3.
\]
Thus to prove that the map $\nu$ 
defined by (\ref{Eq(edge_numbering)}) is well-defined, 
it suffices to see that 
$\varphi(\gamma)=3$ for any closed path 
$\gamma=(e_0,e_1,\dots,e_n=e_0)$. 
Notice that $\varphi$ has the same image after removing a 
``back-tracking'' part, that is, 
if $\gamma=(e_0,\dots,e_{i-1},e_i,e_{i+1},\dots,e_0)$ 
contains a triplet of mutually adjacent edges 
$e_{i-1}$, $e_i$ and $e_{i+1}$, then
\[
	\varphi(\gamma)
	\equiv\varphi(e_0,\dots,e_{i-2},e_{i+2},\dots,e_0)
	\pmod 3,
\]
and if $\gamma=(e_0,\dots,e_{i-1},e_i,e_{i+1},\dots,e_0)$ 
satisfies $e_{i-1}=e_{i+1}$, then
\[
	\varphi(\gamma)
	\equiv\varphi(e_0,\dots,e_{i-1},e_{i+2},\dots,e_0)
	\pmod 3.
\]
Therefore the restriction 
$\varphi\colon \mathcal{CP}(X,e_0)\rightarrow \{1,2,3\}$ 
of $\varphi$ to $\mathcal{CP}(X,e_0)$, 
the set of closed paths with base edge $e_0$, 
descends to a homomorphism
\[
	\overline{\varphi}\colon
	\pi_1(X,m(e_0))\rightarrow \{1,2,3\},
\]
where $\pi_1(X,m(e_0))$ is the fundamental group of $X$ 
with base point $m(e_0)$. 
Since $\{1,2,3\}=\mathbb{Z}/3\mathbb{Z}$ 
is an abelian group, 
$\overline{\varphi}$ further descends to 
a homomorphism
\[
	\widetilde{\varphi}\colon 
	H_1(X,\mathbb{Z})\rightarrow \{1,2,3\},
\]
where $H_1(X,\mathbb{Z})$ is the $1$-dimensional 
homology group of $X$. 
Now any $\gamma\in H_1(X,\mathbb{Z})$ 
can be written as 
$\gamma=\sum_{f\colon \text{face of $X$}}a_f\partial f$, 
where $a_f\in \mathbb{Z}$ and 
$\partial f$ is the cycle consisting of 
edges around $f$. 
Our assumption implies that 
$\widetilde{\varphi}(\partial f)=3$ 
for any face $f$ of $X$. 
Hence we conclude that $\widetilde{\varphi}\equiv 3$, 
which implies that $\varphi\equiv 3$ 
on $\mathcal{CP}(X,e_0)$. 
\end{proof}
A relation between (F) and (C) is stated as follows. 
\begin{Prop}
\label{Prop(F_implies_C)}
Let $X$ be a $3$-valent plane graph 
satisfying {\upshape(F)}. 
Then $X$ has a vertex coherent coloring satisfying 
{\upshape(C-i)}, 
{\upshape(C-ii)} and 
{\upshape(C-iii)}. 
\end{Prop}
\begin{proof}
Let $p_0\in V(X)$ be an arbitrary fixed vertex 
and color it black. 
Every vertex which is accessible by either 
turning left twice or turning right twice 
from a black vertex 
is, one after another, colored in black 
until no more vertices can be colored in black. 
The remaining vertices are colored in white. 
Now we have to check that 
(C-i) and (C-ii) are satisfied 
(while (C-iii) is necessarily satisfied). 
It is easily seen that 
a white vertex is adjacent to at least one black vertex; 
otherwise, all vertices of $X$ must be white. 
It is also easily checked that 
if a white vertex is adjacent to 
two or more black vertices, 
then two other black vertices are necessarily adjacent 
somewhere else. 
So, it suffices to show that 
any pair of black vertices cannot be adjacent. 
\begin{figure}[htbp]
	\centering
	\includegraphics[height=3.5cm]%
	{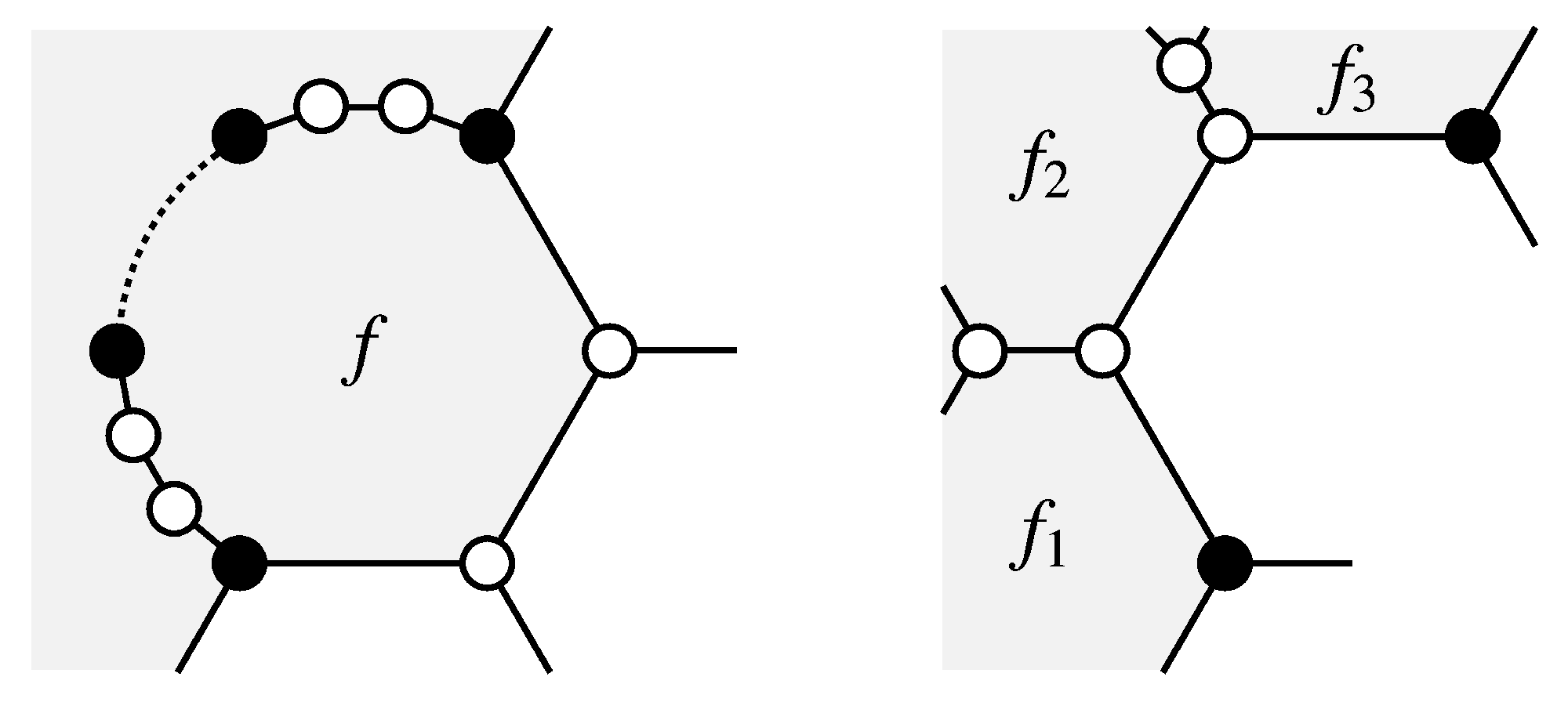}
	\caption{%
	A Part of $\gamma$ 
	where $\gamma$ turns left twice (left) 
	and a part of $\gamma$ 
	where $\gamma$ turns right twice (right). 
	The gray regions are bounded ones 
	surrounded by $\gamma$. 
	}
	\label{Fig(reducing_faces)}
\end{figure}
Suppose that there is a pair of adjacent black vertices, 
say $p,q\in V(X)$. 
From our way of the coloring, 
there is a path $\gamma$ from $p$ to $q$ 
which is a sequence of either twice turning left or 
twice turning right between black vertices. 
Then $\gamma\cup (q,p)$ is a closed path, 
which surrounds a finitely many faces, 
say $f_1,f_2,\dots,f_n$, after removing back-trackings. 
Now if $n=1$, then $\gamma$ consists of a circuit on 
the boundary $\partial f_1$ of a face $f_1$ 
and of some back-trackings 
with black base points on $\partial f_1$, 
which is a contradiction because 
the total of $\tau$ defined by (\ref{Eq(tau(e,e'))}) 
is $0\pmod 3$ after the crossing just prior to 
a lap of $\gamma\cup (q,p)$. 
So assume that $n\geq 2$. 
There are just two possibilities 
of paths along the boundary of $\bigcup_{i=1}^nf_i$ 
connecting a pair of black vertices with distance $3$, 
as indicated in Figure~\ref{Fig(reducing_faces)}. 
In either case, 
we can replace $\gamma\cup (q,p)$ by a closed path 
which does not surround a face $f_i$ 
(by ignoring back-trackings), 
and is still a sequence of either twice turning left or 
twice turning right between black vertices. 
Therefore the conclusion for the case where $n\geq 2$ 
can be deduced from the discussion 
given for the case $n=1$. 
\end{proof}
\begin{Exs}\label{Ex(F-CN-C)}
\mbox{}
\begin{enumerate}
	\item 
	The tetrahedron and 
	any of its Goldberg-Coxeter constructions 
	satisfy all the conditions above. 
	\item 
	$\GC_{2,0}(X)$ for \emph{any} 
	$3$-valent plane graph $X$ always satisfies 
	(C-i), (C-ii) and (C-iii); 
	indeed, we just have to color 
	only the ``center'' of each $(2,0)$-cluster black, 
	and the others white. 
	\begin{figure}[htbp]
		\centering
		\includegraphics[height=3.5cm]{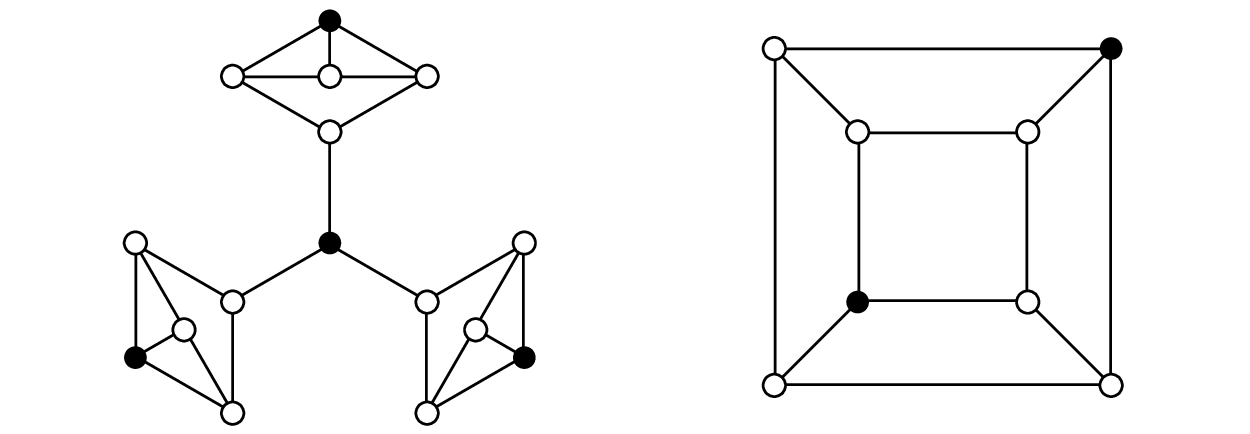}
		\caption{Left coloring satisfies 
		(C-i), (C-ii) and (C-iii). 
		Right coloring on the cube satisfies 
		(C-i) and (C-ii) 
		but does not satisfy (C-iii). }
		\label{Fig(ex_of_C)}
	\end{figure}
	\item \label{Item(GC_1,1_satisfies_C)}
	$\GC_{1,1}(X)$ for \emph{any} 
	$3$-valent plane graph $X$ also always satisfies
	(C-i), (C-ii) and (C-iii); 
	indeed, we just have to color 
	in accordance with the rule 
	shown in Figure~\ref{Fig(coloring_for_GC_11)}. 
	\begin{figure}[htbp]
		\centering
		\includegraphics[height=3.5cm]%
		{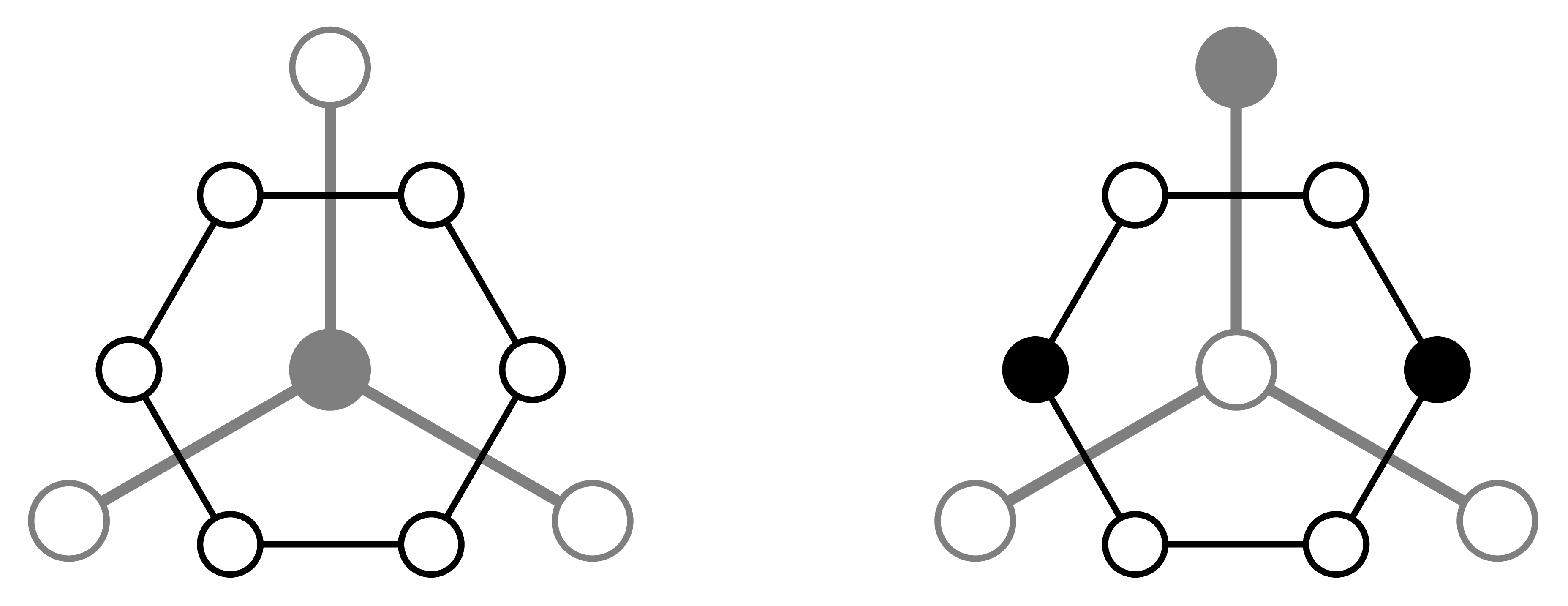}
		\caption{%
		Colorings for the $(1,1)$-cluster 
		around a black (gray in this figure) vertex 
		of $X$ (left) and a white one of $X$ (right). 
		(The gray graphs represent for $X$, 
		while black ones for $\GC_{1,1}(X)$.\ )%
		}
		\label{Fig(coloring_for_GC_11)}
	\end{figure}
	\item 
	The cube satisfies (C-i) and (C-ii) 
	but does not satisfy (C-iii) 
	(see Figure~\ref{Fig(ex_of_C)}), nor, of course, (F). 
	\item 
	The dodecahedron satisfies 
	none of the conditions above. 
\end{enumerate}
\end{Exs}
%
\section{Two comparisons of the eigenvalues}
\label{Sec(comp_of_ev)}
In this section 
we give two kinds of comparisons of the eigenvalues, 
one is that between the eigenvalues of $X$ 
and those of $\GC_{k,l}(X)$, 
and the other is that between 
the eigenvalues of the $(k,0)$-cluster 
and those of $\GC_{k,0}(X)$. 
The former comparison provides the proof of 
Theorems~\ref{Thm(comp_with_X)}, 
and the latter is used in the proof of 
Theorem~\ref{Thm(o(k^2)_evs)}. 
Throughout this section, let $k$ and $l$ be 
integers satisfying $k\geq l\geq 0$ and $k\neq 0$ 
in consideration of 
Proposition~\ref{Prop(properties_of_GC)}%
~(\ref{Item(isom_class_for_GC)}). 
\subsection{The case where $X$ is $3$-valent}
\begin{proof}[Proof of\/ 
{\upshape(\ref{Eq(lambda(GC(X))<lambda(X))})} and 
{\upshape(\ref{Eq(lambda(GC(X))>6-lambda(X))})} in
Theorem~{\upshape\ref{Thm(comp_with_X)}}]
Let us denote by $V(p):=V(\triangle_{k,l}(p))$ 
for a fixed pair $(k,l)$. 
Let $p\in V(X)$, $q\in N_X(p)$ and set
\begin{equation}
	\begin{aligned}
		V_0(p)
		:={}&
		\left\{
		x\in V(p)\setmid
		N_{X'}(x)\subseteq V(p)
		\right\},
		\\
		V_i^q(p)
		:={}&
		\left\{
		x\in V(p)\setmid
		\card{N_{X'}(x)\cap V(q)}=i
		\right\}
		\quad
		(i=1,2).
	\end{aligned}
	\label{Eq(V_0(p)&V_i^q(p))}
\end{equation}
Note that, for any $x\in V(p)$ and $q\in N_X(p)$, 
there are at most two edges emanating 
from $x$ to $V(q)$. 
Since there is nothing to discuss 
when $(k,l)=(1,0)$, we only consider the other cases. 
Let $c=1/\sqrt{\card{V(p)}}=1/\sqrt{k^2+kl+l^2}$ and 
define a linear map 
$Q\colon \mathbb{C}^{V(X)}\rightarrow \mathbb{C}^{V(X')}$ 
for $f\in \mathbb{C}^{V(X)}$ and for $x\in V(p)$ by
\begin{equation}
	(Qf)(x):=cf(p).
\label{Eq(Qf=cuf)}
\end{equation}
The transpose 
$\transp{Q}\colon \mathbb{C}^{V(X')}
\rightarrow \mathbb{C}^{V(X)}$ 
of $Q$ is then written as
\[
	(\transp{Q}g)(p)
	=c\sum_{x\in V(p)}g(x)
\]
for $g\in \mathbb{C}^{V(X')}$ and $p\in V(X)$. 
It then follows that for any $f\in \mathbb{C}^{V(X)}$ 
and for any $p\in V(X)$, 
\[
	(\transp{Q}Qf)(p)
	=c\sum_{x\in V(p)}(Qf)(x)
	=c^2\sum_{x\in V(p)}f(p)
	=f(p),
\]
that is $\transp{Q}Q=\id_{\mathbb{C}^{V(X)}}$. 
Also, for arbitrary $f\in \mathbb{C}^{V(X)}$,
\begin{align*}
	(\transp{Q}\Delta_{X'}Qf)(p)
	={}&
	c\sum_{x\in V(p)}(\Delta_{X'}Qf)(x)
	\\
	={}&
	c\sum_{x\in V(p)}
	\left\{
	3(Qf)(x)-\sum_{y\in N_{X'}(x)}(Qf)(y)
	\right\}
	\\
	={}&
	3c^2\card{V(p)}f(p)
	-c\sum_{x\in V_0(p)}\sum_{y\in N_{X'}(x)}(Qf)(y)
	-c\sum_{x\in V(p)\setminus V_0(p)}
	\sum_{y\in N_{X'}(x)}(Qf)(y).
\end{align*}
The second term equals $-3c^2\card{V_0(p)}f(p)$ 
and the third term is computed as
\begin{align*}
	c\sum_{x\in V(p)\setminus V_0(p)}
	\sum_{y\in N_{X'}(x)}(Qf)(y)
	={}&
	c\sum_{q\in N_X(p)}
	\sum_{x\in V_1^q(p)}
	\sum_{y\in N_{X'}(x)}(Qf)(y)
	+c\sum_{q\in N_X(p)}
	\sum_{x\in V_2^q(p)}
	\sum_{y\in N_{X'}(x)}(Qf)(y)
	\\
	={}&
	c^2\sum_{q\in N_X(p)}
	\card{V_1^q(p)}
	\left(2f(p)+f(q)\right)
	+c^2\sum_{q\in N_X(p)}
	\card{V_2^q(p)}
	\left(f(p)+2f(q)\right)
	\\
	={}&
	3c^2(2\card{V_1^q(p)}+\card{V_2^q(p)})f(p)
	+c^2(\card{V_1^q(p)}+2\card{V_2^q(p)})
	\sum_{q\in N_X(p)}f(q),
\end{align*}
where the last equality follows from 
the symmetry of $\triangle_{k,l}(p)$. 
Therefore we obtain
\begin{align*}
	(\transp{Q}\Delta_{X'}Qf)(p)
	={}&
	c^2(\card{V_1^q(p)}+2\card{V_2^q(p)})
	(\Delta_Xf)(p)
	\\
	={}&
	\frac{\mu(k,l)}{k^2+kl+l^2}(\Delta_Xf)(p),
\end{align*}
where $\mu(k,l)$ is the number of edges in $X'$ 
connecting two clusters 
and depends only on $k$ and $l$. 
It is easily proved that 
$\mu(k,0)=k$ and $\mu(k,k)=3k$. 
To estimate $\mu(k,l)$ when $k>l>0$, 
let us estimate the number of edges 
crossing the edge $E=0z$. 
Notice first that there is at most one crossing edge 
emanating from an upward triangle $\triangle(a,b)$, 
and that there are at most two crossing edge 
emanating from a downward triangle $\dtriangle(a,b)$. 
For $c\in \mathbb{Z}$, 
``the zigzag path'' 
which is obtained by joining the barycenters 
of $\dtriangle(a,b)$, $\triangle(a,b)$ 
and $\dtriangle(a+1,b-1)$ 
for all $a,b\in \mathbb{Z}$ with $a+b=c$ 
crosses the edge $E=0z$ exactly once 
provided $0\leq c\leq k+l-1$ 
and does not cross $E$ otherwise. 
Also, the line passing through $a\in \mathbb{Z}$ 
with slant $1+\omega$ 
crosses $E$ exactly once 
provided $0\leq a\leq k-l$ 
and does not cross $E$ otherwise. 
Therefore the number of edges crossing $E$ is at most 
$k+l+(k-l-2)=2k-2$
(see Figure~\ref{Fig(mu(k,l)_example)} 
for an example).
\par
(\ref{Eq(lambda(GC(X))<lambda(X))}) 
of Theorem~\ref{Thm(comp_with_X)} 
now immediately follows from the following.
\begin{Thm}[Interlacing property, 
see for example \cite{MR2571608}]
\label{Thm(interlacing)}
Let $Q$ be a real $n\times m$ matrix satisfying 
$\transp{Q}Q=I_m$ 
and $A$ be a real symmetric $n\times n$ matrix. 
If the eigenvalues of $A$ and $\transp{Q}AQ$ are
\[
	\nu_1(A)\leq \nu_2(A)\leq \dots \leq \nu_n(A),
	\quad
	\nu_1(\transp{Q}AQ)\leq \nu_2(\transp{Q}AQ)\leq 
	\dots \leq \nu_m(\transp{Q}AQ),
\]
respectively, then
\[
	\nu_j(A)
	\leq \nu_j(\transp{Q}AQ)
	\leq \nu_{n-m+j}(A)
	\quad
	(j=1,2,\dots,m).
\]
\end{Thm}
\begin{figure}[htbp]
	\centering
	\includegraphics[width=.5\textwidth]%
	{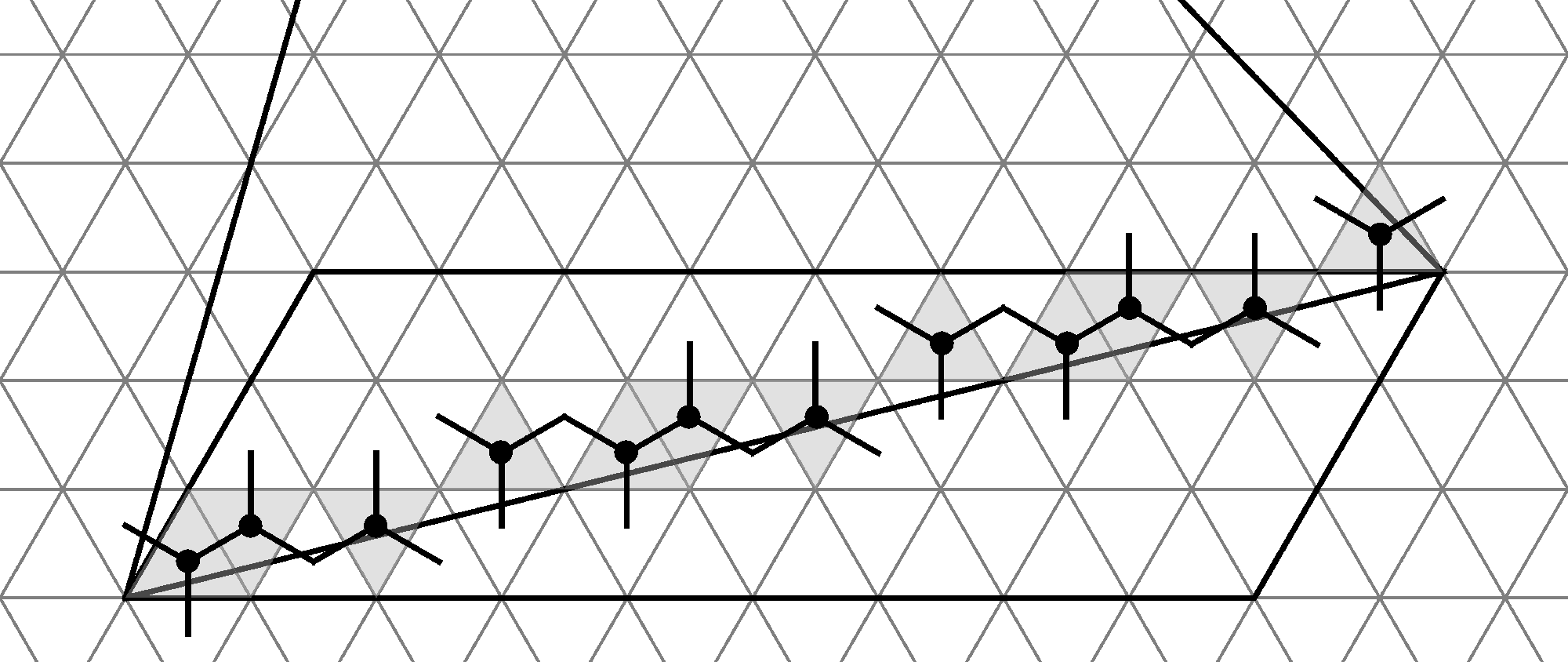}
	\caption{$(k,l)=(9,3)$. 
	$15$ edges cross the edge $E=0z$, ($z=9+3\omega$). }
	\label{Fig(mu(k,l)_example)}
\end{figure}
(\ref{Eq(lambda(GC(X))>6-lambda(X))}) 
is an immediate consequence from 
(\ref{Eq(lambda(GC(X))<lambda(X))}) and 
Proposition~\ref{Prop(GC_bipartite)}.
\end{proof}
The eigenvalues of $\GC_{k,0}(X)$ 
are estimated, 
independently of the graph structure of $X$, 
also by those of the $(k,0)$-cluster as follows. 
\begin{Thm}
\label{Thm(comp_with_3-vlnt_clstr)}
Let $X$ be a $3$-valent graph 
satisfying the same assumptions 
as in Theorem~{\upshape\ref{Thm(comp_with_X)}}, 
and $\nu_1(k)\leq \nu_2(k)\leq \cdots\leq \nu_{k^2}(k)$ 
{\upshape(}resp.\ 
$0=\lambda_1(k)\leq \lambda_2(k)\leq
\cdots\leq \lambda_{k^2}(k)${\upshape)} 
be the eigenvalues of the adjacency matrix 
{\upshape(}resp.\ of the Laplacian{\upshape)} 
of the $3$-valent $(k,0)$-cluster. 
Then for $j=1,2,\dots,k^2$,
\begin{gather}
	\lambda_j(\GC_{k,0}(X))
	\leq
	3-\nu_{k^2-j+1}(k),
	\label{Eq(lambda(X_k)<3-lambda(C_k))}
	\\
	\lambda_{\card{V(\GC_{k,0}(X))}-j+1}(\GC_{k,0}(X))
	\geq
	3-\nu_j(k).
	\label{Eq(lambda(X_k)>3-lambda(C_k))}
\end{gather}
Moreover, we have
\begin{gather}
	\lambda_i(\GC_{k,0}(X))
	\leq
	\lambda_t(k)+\delta_{k^2-t+i}(k),
	\quad
	\text{for $1\leq i\leq t\leq k^2$},
	\label{Eq(lambda(X_k)<lambda(C_k)+delta)}
	\\
	\lambda_{\card{V(\GC_{k,0}(X))}-j+1}(\GC_{k,0}(X))
	\geq
	\lambda_{k^2-s+1}(k)+\delta_{1+s-j}(k),
	\quad
	\text{for $1\leq j\leq s\leq k^2$},
	\label{Eq(lambda(X_k)>lambda(C_k)+delta)}
\end{gather}
where $\delta_j(k)$ is given as
\[
	\delta_j(k)
	=\left\{
	\begin{aligned}
		& 0,
		& & \text{for $j=1,2,\dots,k^2-3k+3$},
		\\
		& 1,
		& & \text{for $j=k^2-3k+4,\dots,k^2-3$},
		\\
		& 2,
		& & \text{for $j=k^2-2,k^2-1,k^2$}
	\end{aligned}
	\right.
\]
\end{Thm}
\begin{proof}
Let $p\in V(X)$ be fixed and 
let $\triangle(k)=\triangle(p)$ 
be the $(k,0)$-cluster, 
which is considered as a subgraph of $X_k=\GC_{k,0}(X)$. 
Let us define a linear map 
$Q\colon \mathbb{C}^{V(X_k)}\rightarrow
\mathbb{C}^{V(\triangle(k))}$ 
by
\[
	(Qf)(x):=
	\left\{
	\begin{aligned}
		& f(x),
		& & \text{if $x\in V(\triangle(k))$},
		\\
		& 0,
		& & \text{if $x\not\in V(\triangle(k))$}
	\end{aligned}
	\right.
\]
for $f\in \mathbb{C}^{V(X_k)}$ and $x\in V(X_k)$. 
Then a simple computation shows 
$\transp{Q}Q=\id_{\mathbb{C}^{V(X_k)}}$ 
and $\transp{Q}A_{X_k}Q=A_{\triangle(k)}$, 
where $A$'s denote the adjacency matrices. 
By noting that $X_k=\GC_{k,0}(X)$ is 
a $3$-regular graph, 
the interlacing property 
(Theorem~\ref{Thm(interlacing)}) 
proves 
(\ref{Eq(lambda(X_k)<3-lambda(C_k))}) and 
(\ref{Eq(lambda(X_k)>3-lambda(C_k))}). 
Since
\[
	\transp{Q}\Delta_{X_k}Q
	=\Delta_{\triangle(k)}
	+(3\id_{\mathbb{C}^{V(\triangle(k))}}-D_k),
\]
where $D_k\colon \mathbb{C}^{V(\triangle(k))}
\rightarrow \mathbb{C}^{V(\triangle(k))}$ 
is defined as $(Df)(x):=\deg(x)$ 
for $f\in \mathbb{C}^{V(\triangle(k))}$ 
and $x\in V(\triangle(k))$, 
Combining the Courant-Weyl inequality (cf. \cite[Theorem 1.3.15]{MR2571608})
and the interlacing property proves 
(\ref{Eq(lambda(X_k)<lambda(C_k)+delta)}) 
and (\ref{Eq(lambda(X_k)>lambda(C_k)+delta)}). 
\end{proof}
\subsection{The case where $X$ is $4$-valent}
The proof of (\ref{Eq(lambda(GC(X))<lambda(X))}) 
for the $4$-valent case is almost same as 
that for $3$-valent case, and let us omit it. 
The comparison between 
the eigenvalues of a $(k,0)$-cluster 
and those of $\GC_{k,0}(X)$ 
for the $4$-valent case is stated as follows.
\begin{Thm}
\label{Thm(comp_with_4-vlnt_clstr)}
Let $X$ be a $4$-valent graph 
satisfying the same assumptions 
as in Theorem~{\upshape\ref{Thm(comp_with_X)}}, 
and $\nu_1(k)\leq \nu_2(k)\leq \cdots\leq \nu_{k^2}(k)$ 
{\upshape(}resp.\ 
$0=\lambda_1(k)\leq \lambda_2(k)\leq
\cdots\leq \lambda_{k^2}(k)${\upshape)} 
be the eigenvalues of the adjacency matrix 
{\upshape(}resp.\ of the Laplacian{\upshape)} 
of the $4$-valent $(k,0)$-cluster. 
Then for $j=1,2,\dots,k^2$,
\begin{gather*}
	\lambda_j(\GC_{k,0}(X))
	\leq
	4-\nu_{k^2-j+1}(k),
	\\
	\lambda_{\card{V(\GC_{k,0}(X))}-j+1}(\GC_{k,0}(X))
	\geq
	4-\nu_j(k),
\end{gather*}
Moreover, we have
\begin{gather*}
	\lambda_i(\GC_{k,0}(X))
	\leq
	\lambda_t(k)+\delta_{k^2-t+i}(k),
	\quad
	\text{for $1\leq i\leq t\leq k^2$},
	\\
	\lambda_{\card{V(\GC_{k,0}(X))}-j+1}(\GC_{k,0}(X))
	\geq
	\lambda_{k^2-s+1}(k)+\delta_{1+s-j}(k),
	\quad
	\text{for $1\leq j\leq s\leq k^2$},
\end{gather*}
where $\delta_j(k)$ is given as
\[
	\delta_j(k)
	=\left\{
	\begin{aligned}
		& 0,
		& & \text{for $j=1,2,\dots,k^2-4k+4$},
		\\
		& 1,
		& & \text{for $j=k^2-4k+5,\dots,k^2-4$},
		\\
		& 2,
		& & \text{for $j=k^2-3,k^2-2,k^2-1,k^2$}.
		\\
	\end{aligned}
	\right.
\]
\end{Thm}
Since the proof of 
this theorem is again almost same as 
that for the $3$-valent case, 
let us omit it.
%
\section{Eigenvalues of the $(k,0)$-cluster}
\label{Sec(ev_of_clstr)}
In this section we shall find all the eigenvalues 
of a $(k,0)$-cluster to prove 
Theorem~\ref{Thm(o(k^2)_evs)}. 
Since the $(k,0)$-clusters are, as abstract graphs, 
isomorphic to each other, 
fixing a vertex $p\in V(X)$, 
we may denote it by 
$\triangle(k):=\triangle_{k,0}(p)
=\overline{\triangle}_{k,0}(p)$ 
or 
$\square(k):=\square_{k,0}(p)
=\overline{\square}_{k,0}(p)$. 
\subsection{The case where $X$ is $3$-valent}
\begin{Def}
\label{Def(D_3-inv&alt_ev)}
$\lambda\geq 0$ is called a 
\emph{$D_3$-invariant eigenvalue} 
(resp.\ \emph{$D_3$-alternating eigenvalue})
for a $(k,0)$-cluster $\triangle(k)$ 
if there exists a non-zero function 
$u\colon V(\triangle(k))\rightarrow \mathbb{C}$, called a 
\emph{$D_3$-invariant eigenfunction}
(resp.\ \emph{$D_3$-alternating eigenfunction}), 
with the following properties. 
\begin{enumerate}[label=(\roman*),ref=\roman*]
	\item 
	$u$ solves $(\Delta_{\triangle(k)}u)(x)=\lambda u(x)$ 
	for $x\in V(\triangle(k))$
	\item 
	$u(\sigma x)=u(x)$
	(resp.\ $u(\sigma x)=\sgn(\sigma)u(x)$)
	for $x\in V(\triangle(k))$, 
	where 
	$\sigma\colon \triangle(k)\rightarrow \triangle(k)$ 
	is an element of the dihedral group $D_3$
	and $\sgn(\sigma)$ denotes its signature.
\end{enumerate}
\end{Def}
\begin{Rems}
\label{Rem(inv&alt)}
(1) The following remark shall be repeatedly used 
in the sequel: 
by assigning the same function $u$ 
to the other clusters, 
we have a global function 
$u\colon \GC_{k,0}(X)\rightarrow \mathbb{C}$, 
which is an eigenfunction of $\Delta_{\GC_{k,0}(X)}$ 
with eigenvalue $\lambda$; indeed, 
(i) $\Delta_{\triangle(k)}u=\lambda u$ 
is equivalent to a Neumann problem:
\begin{equation}
	\left\{
	\begin{aligned}
		& (\Delta_{\GC_{k,0}(X)}u)(x)
		=\lambda u(x),
		& & \text{for $x\in V(p)$},
		\\
		& u(y)-u(x)=0,
		& & \text{for $x\in V(p)\setminus V_0(p)$ 
		and 
		$y\in N_{\GC_{k,0}(X)}(x)\setminus V(p)$}
	\end{aligned}
	\right.
	\label{Eq(Neumann_prob)}
\end{equation}
for some/any $p\in V(X)$.
\par
(2) There is no eigenfunction on a $(k,0)$-cluster
which is both $D_3$-invariant and $D_3$-alternating.
\end{Rems}
Our first task is to find 
all the $D_3$-invariant eigenspaces, 
which proves Theorem~\ref{Thm(arbitrary_lambda)} 
as well as (\ref{Eq(lambda(GC(X))>6-e)}) 
in Theorem~\ref{Thm(comp_with_X)}. 
To this end, 
let us first construct all the eigenfunctions 
on a hexagonal lattice with toroidal boundary condition. 
If we set $\cm:=(1+\omega)/3$, 
where $\omega=e^{\pi i/3}$, then the discrete set
\begin{equation}
	\left\{
	a+b\omega
	\setmid
	a,b\in \mathbb{Z}
	\right\}
	\cup
	\left\{
	\cm+a+b\omega
	\setmid
	a,b\in \mathbb{Z}
	\right\}
	\label{Eq(hex_coord)}
\end{equation}
is naturally regarded as a hexagonal lattice. 
For a fixed $k \in \mathbb{N}$, 
let us consider the equations
\begin{equation}
	\begin{aligned}
		& 3v(a+b\omega)
		-v(\cm+a+b\omega)
		-v(\cm+a-1+b\omega)
		-v(\cm+a+(b-1)\omega)
		=\lambda v(a+b\omega),
		\\
		& 3v(\cm+a+b\omega)
		-v(a+b\omega)
		-v(a+1+b\omega)
		-v(a+(b+1)\omega)
		=\lambda v(\cm+a+b\omega)
	\end{aligned}
	\label{Eq(eigen_eq_on_P(k))}
\end{equation}
for a function $v$ on the parallelogram
\[
	P(k)
	:=
	\left\{
	a+b\omega
	\setmid
	0\leq a,b\leq k-1
	\right\}
	\cup
	\left\{
	\cm+a+b\omega
	\setmid
	0\leq a,b\leq k-1
	\right\},
\]
where $a$ and $b$ in (\ref{Eq(eigen_eq_on_P(k))}) 
are considered modulo $k$, such as
\[
	3v(0)-v(\cm)-v(\cm+k-1)-v(\cm+(k-1)\omega)
	=\lambda v(0)
\]
for the former equation of (\ref{Eq(eigen_eq_on_P(k))}) 
with $a=b=0$. 
So if $v$ solves (\ref{Eq(eigen_eq_on_P(k))}), 
then it gives an eigenfunction with eigenvalue $\lambda$ 
on the finite $3$-valent graph $T(k)$ with $2k^2$ vertices 
obtained by adding edges between 
$a$ and $\cm+a+(k-1)\omega$, 
and between $b\omega$ and $\cm+k-1+b\omega$ 
for each $a,b=0,1,\dots,k-1$.
\par
A simple computation shows that 
an eigenvalue is of the form
\begin{equation}
	\lambda_{s,t}^{\pm}
	=\lambda_{s,t}^{\pm}(k)
	=
	3\pm\sqrt{%
	3+2\cos\frac{2\pi s}{k}
	+2\cos\frac{2\pi t}{k}
	+2\cos\frac{2\pi (s-t)}{k}
	},
	\label{Eq(3-vlnt_lambda_st^pm)}
\end{equation}
whose corresponding eigenfunction is given as
\[
	\left\{
	\begin{aligned}
		v_{s,t}^{\pm}(a+b\omega)
		={}&
		e^{2\pi i(sa+tb)/k},
		\\
		v_{s,t}^{\pm}(\cm+a+b\omega)
		={}&
		\frac{1}{3-\lambda_{s,t}^{\pm}}
		v_{s,t}^{\pm}(a+b\omega)
		\left(
		1+e^{2\pi is/k}+e^{2\pi it/k}
		\right)
	\end{aligned}
	\right.
\]
($a+b\omega,\cm+a+b\omega\in P(k)$) 
for $s,t=0,1,\dots,k-1$, 
unless $\lambda_{s,t}^{\pm}=3$. 
If $\lambda_{s,t}^{\pm}=3$, 
which is possible only if $k\equiv 0\pmod 3$ 
and either $(s,t)=(k/3,2k/3)$ or $(2k/3,k/3)$ 
among the range $0\leq s,t\leq k-1$, then
\[
	\left\{
	\begin{aligned}
		v_{k/3,2k/3}(a+b\omega)
		={}&
		\alpha e^{2\pi i(a+2b)/3},
		\\
		v_{k/3,2k/3}(\cm+a+b\omega)
		={}&
		\alpha' e^{2\pi i(a+2b)/3},
	\end{aligned}
	\right.
	\quad
	\left\{
	\begin{aligned}
		v_{2k/3,k/3}(a+b\omega)
		={}&
		\beta e^{2\pi i(2a+b)/3},
		\\
		v_{2k/3,k/3}(\cm+a+b\omega)
		={}&
		\beta' e^{2\pi i(2a+b)/3},
	\end{aligned}
	\right.
\]
where $\alpha,\alpha',\beta,\beta'\in \mathbb{C}$ 
are arbitrary, 
both define eigenfunctions for the eigenvalue $3$.
\par
We now consider the following three maps 
defined on the hexagonal lattice:
\begin{itemize}
	\item 
	the rotation around $(k-1)(1+\omega)/3$ 
	by $2\pi/3$: 
	$\left\{
	\begin{aligned}
		& a+b\omega\mapsto k-a-b-1+a\omega,
		\\
		& \cm+a+b\omega\mapsto \cm+k-a-b-2+a\omega,
	\end{aligned}
	\right.$
	\item 
	the reflection along the long diagonal line 
	of $P(k)$: 
	$\left\{
	\begin{aligned}
		& a+b\omega\mapsto b+a\omega,
		\\
		& \cm+a+b\omega\mapsto \cm+b+a\omega,
	\end{aligned}
	\right.$
	\item 
	and the reflection along the short one: 
	$\left\{
	\begin{aligned}
		& a+b\omega\mapsto \cm+k-b-1+(k-a-1)\omega,
		\\
		& \cm+a+b\omega\mapsto k-b-1+(k-a-1)\omega.
	\end{aligned}
	\right.$
\end{itemize}
These maps define, 
by considering $a$ and $b$ modulo $k$, 
automorphisms on $T(k)$, 
and generate the dihedral group $D_6$ of order $12$. 
As is easily confirmed, 
taking the average 
$\sum_{\sigma\in D_6}\sigma f$ 
(resp.\ $\sum_{\sigma\in D_6}\sgn(\sigma)\sigma f$) 
for $f\in \mathbb{C}^{P(k)}$ defines 
a projection onto the $D_3$-invariant eigenspaces 
(resp.\ $D_3$-alternating eigenspaces)
for the $(k,0)$-cluster, 
where $\sgn(\sigma)$ is the number modulo $2$ of 
the reflections along 
the \emph{long} diagonal line of $P(k)$ 
in an expression of $\sigma$. 
Now we set, for $s,t=0,1,\dots,k-1$,
\begin{equation}
	u_{s,t}^{\pm}
	:=
	\sum_{\sigma\in D_6}\sigma v_{s,t}^{\pm}
	\quad
	\text{and}%
	\quad
	w_{s,t}^{\pm}
	:=
	\sum_{\sigma\in D_6}\sgn(\sigma)\sigma v_{s,t}^{\pm},
	\label{Eq(u_st^pm&w_st^pm)}
\end{equation}
which respectively give a $D_3$-invariant eigenfunction 
and a $D_3$-alternating eigenfunction on $\triangle(k)$ 
unless they identically vanish on $\triangle(k)$. 
Note that these functions respectively generate 
the space of $D_3$-invariant eigenfunctions 
and the one of $D_3$-alternating eigenfunctions 
because they define functions also on $T(k)$. 
The following Lemma~\ref{Lem(all_inv&alt_evs)} 
explicitly tells us when 
$u_{s,t}^{\pm}$ and $w_{s,t}^{\pm}$ vanish.
\par
\begin{Lem}
\label{Lem(all_inv&alt_evs)}
Let $0\leq s,t\leq k-1$. 
$u_{s,t}^{\pm}\equiv 0$ if and only if 
$u_{s,t}^{\pm}$ is one of the following:
\begin{itemize}[itemsep=5pt]
	\item 
	$u_{s,k-s}^{+}$ or $u_{k-s,s}^{+}$ 
	for $1\leq s<k/3$;
	\item 
	$u_{s,k-s}^{-}$ 
	for $k/3<s<2k/3$;
	\item 
	$u_{s,2s}^{+}$ or $u_{2s,s}^{+}$ 
	for $0\leq s<k/3$;
	\item 
	$u_{s,2s}^{-}$ or $u_{2s,s}^{-}$ 
	for $k/3<s<k/2$;
	\item 
	$u_{s,2s-k}^{-}$ or $u_{2s-k,s}^{-}$ 
	for $k/2\leq s<2k/3$;
	\item 
	$u_{s,2s-k}^{+}$ or $u_{2s-k,s}^{+}$ 
	for 
	$2k/3<s\leq k-1$.
\end{itemize}
On the other hand, 
$w_{s,t}^{\pm}\equiv 0$ if and only if 
$w_{s,t}^{\pm}$ is one of the following:
\begin{itemize}[itemsep=5pt]
	\item 
	$w_{s,0}^{\pm}$ or $w_{0,t}^{\pm}$ 
	for $0\leq s,t\leq k-1$;
	\item 
	$w_{s,s}^{\pm}$ 
	for $0\leq s\leq k-1$;
	\item 
	$w_{s,k-s}^{+}$ or $w_{k-s,s}^{+}$ 
	for $1\leq s<k/3$
	\item 
	$w_{s,k-s}^{-}$ 
	for $k/3<s<2k/3$;
	\item 
	$w_{s,2s}^{+}$ or $w_{2s,s}^{+}$ 
	for $0\leq s<k/3$;
	\item 
	$w_{s,2s}^{-}$ or $w_{2s,s}^{-}$ 
	for 
	$k/3<s<k/2$;
	\item 
	$w_{s,2s-k}^{-}$ or $w_{2s-k,s}^{-}$ 
	for $k/2\leq s<2k/3$;
	\item 
	$w_{s,2s-k}^{+}$ or $w_{2s-k,s}^{+}$ 
	for $2k/3<s\leq k-1$.
\end{itemize}
In particular the associated eigenvalues 
$\lambda_{s,t}^{\pm}$ other than the above lists give 
$D_3$-invariant or 
$D_3$-alternating eigenvalues, respectively.
\end{Lem}
\begin{proof}
The proof uses explicit expression of 
$u_{s,t}^{\pm}$ and $w_{s,t}^{\pm}$ 
via the coordinate (\ref{Eq(hex_coord)}). 
\par
A direct computation shows that 
\[
	\begin{aligned}
		&u_{s,t}^{\pm}(0)
		\left(1+e^{4\pi is/k}+e^{4\pi it/k}\right)
		\left(1+e^{-2\pi is/k}+e^{-2\pi it/k}\right)
		\\
		&\pm
		u_{s,t}^{\pm}(\cm)
		\left(1+e^{2\pi is/k}+e^{2\pi it/k}\right)
		\left\vert
		\left(1+e^{2\pi is/k}+e^{2\pi it/k}\right)
		\right\vert
		\\
		={}&
		-32i
		\sin\frac{\pi(s+t)}{k}
		\sin\frac{\pi(2s-t)}{k}
		\sin\frac{\pi(s-2t)}{k},
	\end{aligned}
\]
from which the above list for $u_{s,t}^{\pm}$ is obtained. 
\par
A direct computation shows that 
\[
	\begin{aligned}
		&w_{s,t}^{\pm}(1)
		\left(
		1+e^{2\pi is/k}+e^{2\pi it/k}+e^{2\pi i(s+t)/k}
		+e^{2\pi i(s-t)/k}+e^{2\pi i(t-s)/k}
		\right)
		\\
		&\pm
		w_{s,t}^{\pm}(\cm+1)
		\left(
		1+e^{2\pi is/k}+e^{2\pi it/k}
		\right)
		\\
		={}&
		64
		\sin\frac{\pi s}{k}
		\sin\frac{\pi t}{k}
		\sin\frac{\pi(s-t)}{k}
		\sin\frac{\pi(s+t)}{k}
		\sin\frac{\pi(2s-t)}{k}
		\sin\frac{\pi(s-2t)}{k},
	\end{aligned}
\]
from which the above list for $w_{s,t}^{\pm}$ is obtained. 
\end{proof}
\begin{proof}[Proof of\/ 
{\upshape(\ref{Eq(lambda(GC(X))>6-e)})} 
in Theorem~{\upshape\ref{Thm(comp_with_X)}}]
Let us prove that 
if $\lambda\geq 0$ is a $D_3$-invariant eigenvalue 
for the $(k,0)$-cluster, then
\begin{equation}
	\lambda
	\leq
	\lambda_{\card{V(\GC_{k,0}(X))}-i+1}(\GC_{k,0}(X))
	\label{Eq(inv_lambda<lambda(GC(X)))}
\end{equation}
holds for $i=1,2,\dots,\card{V(X)}$. 
\par
Let $u\colon \GC_{k,0}(X)\rightarrow \mathbb{C}$ 
be an eigenfunction for the eigenvalue $\lambda$ 
which is obtained, as was explained 
in (1) of Remarks~\ref{Rem(inv&alt)}, 
from a $D_3$-invariant eigenfunction 
on the $(k,0)$-cluster. 
We may assume that $\sum_{x\in V(p)}u(x)^2=1$, 
so that $\transp{Q}Q=\id_{\mathbb{C}^{V(X)}}$. 
Replacing $c$ in (\ref{Eq(Qf=cuf)}) by $u$, 
after a straightforward computation using (i) and (ii) 
in Definition~\ref{Def(D_3-inv&alt_ev)} for $u$, 
we can obtain the following equality:
\begin{equation}
	(\transp{Q}\Delta_{\GC_{k,0}(X)}Qf)(p)
	=
	\left\{
	\sum_{x\in V_1^q(p)}u(x)^2
	+
	2\sum_{x\in V_2^q(p)}u(x)^2
	\right\}
	(\Delta_Xf)(p)
	+\lambda f(p)
	\label{Eq(t^Q_Delta_Q)}
\end{equation}
for any $f\in \mathbb{C}^{V(X)}$ and any $p\in V(X)$, 
where $q\in N_X(p)$ is an adjacent vertex to $p$. 
(\ref{Eq(inv_lambda<lambda(GC(X)))}) is proved 
again from the interlacing property 
(Theorem~\ref{Thm(interlacing)}).
\par
(\ref{Eq(lambda(GC(X))>6-e)}) is an immediate consequence 
of Lemma~\ref{Lem(all_inv&alt_evs)}, which claims that
\[
	\lambda
	=
	\lambda_{1,0}^{+}
	=
	\lambda_{0,1}^{+}
	=
	3+\sqrt{5+4\cos\frac{2\pi}{k}}
\]
is the largest $D_3$-invariant eigenvalue 
for the $(k,0)$-cluster. 
\end{proof}

\begin{proof}[Proof of 
Theorem~{\upshape\ref{Thm(arbitrary_lambda)}}]
It follows from the consequence of 
Lemma~\ref{Lem(all_inv&alt_evs)} that
\begin{itemize}[itemindent=-15pt]
	\item 
	$\lambda_{j,k-j}^{-}
	=3-\sqrt{3+4\cos\frac{2\pi j}{k}+2\cos\frac{4\pi j}{k}}$ 
	for $0\leq j\leq \lceil k/3\rceil-1$;
	\item 
	$\lambda_{\lfloor k/3\rfloor+j,k-\lfloor k/3\rfloor-j}^{+}
	=3+\sqrt{3+4\cos\frac{2\pi(\lfloor k/3\rfloor+j)}{k}
	+2\cos\frac{4\pi(\lfloor k/3\rfloor+j)}{k}}$
	for 
	$1\leq j\leq \lceil 2k/3\rceil-\lfloor k/3\rfloor -1$;
	\item 
	$\lambda_{s,0}^{+}
	=3+\sqrt{5+4\cos\frac{2\pi s}{k}}$ 
	for $1\leq s\leq k-1$
\end{itemize}
are all $D_3$-invariant eigenvalues 
for the $(k,0)$-cluster. 
In the expression of $\lambda_{s,0}^{+}$, set $x = s/k \in [0, 1)$, then 
the function $3+\sqrt{5+4\cos(2\pi x)}$ takes value $[4, 6]$.
Taking large $k$, we may approximate any number $x \in [0, 1)$ by $s/k$ ($0 \le s \le k-1$), 
and an arbitrary real number in $[4, 6]$ is approximated by $\lambda_{s,0}^{+}$.
Similary an arbitrary real number in $[0, 3]$ and $[3, 4]$ 
is approximated by $\lambda_{j,k-j}^{-}$ and $\lambda_{\lfloor k/3\rfloor+j,k-\lfloor k/3\rfloor-j}^{+}$, 
respectively.
Hence 
an arbitrary real number in $[0,6]$ is approximated 
by these values as $k$ tends to infinity.
\end{proof}
In order to prove Theorem~\ref{Thm(o(k^2)_evs)} 
using Theorem~\ref{Thm(comp_with_3-vlnt_clstr)}, 
it suffices to find all the eigenvalues 
of a $(k,0)$-cluster. 
To achieve this, 
we notice that the set of all the eigenvalues of 
a $(k,0)$-cluster contains the set of 
all the $D_3$-invariant and all the $D_3$-alternating 
eigenvalues of the $(3k,0)$-cluster; 
indeed, we have a well-defined injection
\[
	\iota\colon
	\mathbb{C}^{\triangle(k)}
	\rightarrow
	\mathcal{U}_{3k}
	\oplus
	\mathcal{W}_{3k}
\]
which is defined by the foldings like shown 
in Figure~\ref{Fig(patapata_3vlnt)}, 
where $\mathcal{U}_{3k}$ 
(resp.\ $\mathcal{W}_{3k}$) 
denotes the space of $D_3$-invariant 
(resp.\ $D_3$-alternating) 
eigenfunctions on the $(3k,0)$-cluster. 
\begin{figure}[htbp]
	\centering
	\includegraphics[height=5cm]{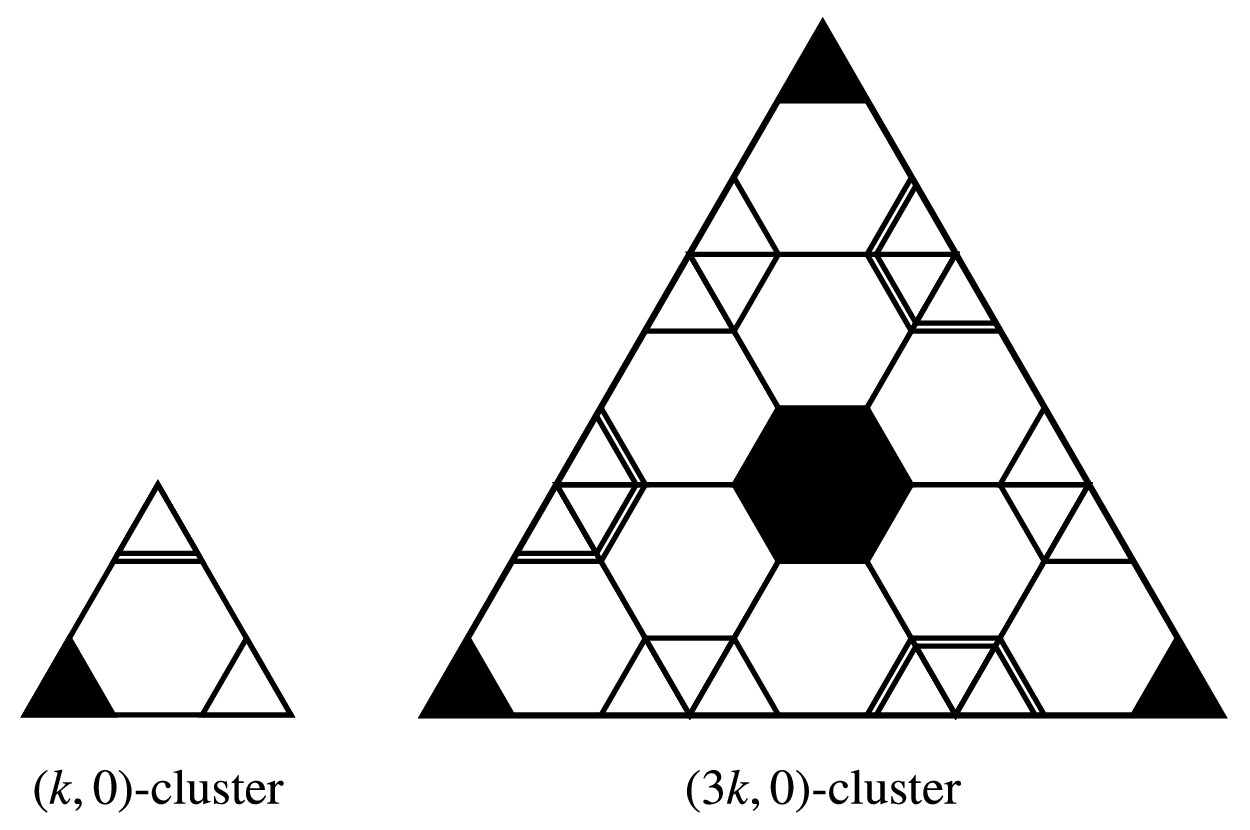}
	\caption{%
	$\triangle(3k)$ is tiled 
	by the foldings of $\triangle(k)$; 
	if the function on $\triangle(k)$ is 
	symmetric (resp.\ antisymmetric) 
	w.r.t.\ the line $a=b$, then the obtained function 
	lies in $\mathcal{U}_{3k}$ (resp.\ $\mathcal{W}_{3k}$).
	}
	\label{Fig(patapata_3vlnt)}
\end{figure}
\begin{Lem}
\label{Lem(patapata)}
$u_{s,t}^{\pm}\in \mathcal{U}_{3k}$ 
{\upshape(}resp.\ 
$w_{s,t}^{\pm}\in \mathcal{W}_{3k}${\upshape)} 
lies in the image of $\iota$ if and only if 
$s+t$ is divisible by $3$. 
Moreover, if both $s$ and $t$ are divisible by $3$, 
then $u_{s,t}^{\pm}|_{\triangle(k)}$ 
{\upshape(}resp.\ 
$w_{s,t}^{\pm}|_{\triangle(k)}${\upshape)} 
is also a $D_3$-invariant 
{\upshape(}resp.\ $D_3$-alternating{\upshape)} 
eigenfunction. 
\end{Lem}
\begin{proof}[Outline of the Proof]
The proof again uses explicit expression 
of $u_{s,t}^{\pm}$ and $w_{s,t}^{\pm}$. 
Note first that $f=u_{s,t}^{\pm}$ or $w_{s,t}^{\pm}$ 
lies in the image of $\iota$ iff.\ 
\begin{equation}
	\begin{aligned}
		f_{s,t}^{\pm}(a+b\omega)
		={}&
		f_{s,t}^{\pm}(\cm+k-b-1+(k-a-1)\omega),
		\\
		f_{s,t}^{\pm}(a+b\omega)
		={}&
		f_{s,t}^{\pm}(\cm+2k-b-1+(2k-a-1)\omega)
	\end{aligned}
	\label{Eq(patapata)}
\end{equation}
for any $a,b$. 
\par
A direct computation similar to that in the proof of 
Lemma~\ref{Lem(all_inv&alt_evs)} shows that 
(\ref{Eq(patapata)}) for $f_{s,t}^{\pm}=u_{s,t}^{\pm}$ 
with $a=b=0$ implies
\[
	\left(1-e^{-2\pi i(s+t)/3}\right)
	\sin\frac{\pi(s+t)}{3k}
	\sin\frac{\pi(2s-t)}{3k}
	\sin\frac{\pi(s-2t)}{3k}
	=0,
\]
which is valid only if $s+t$ is divisible by $3$. 
\par
A direct computation shows that 
(\ref{Eq(patapata)}) for 
$f_{s,t}^{\pm}=w_{s,t}^{\pm}$ 
with $(a,b)=(1,0)$ implies
\[
	\left(1-e^{-2\pi i(s+t)/3}\right)
	\sin^2\frac{\pi s}{3k}
	\sin^2\frac{\pi t}{3k}
	\sin^2\frac{\pi(s-t)}{3k}
	\sin\frac{\pi(s+t)}{3k}
	\sin\frac{\pi(2s-t)}{3k}
	\sin\frac{\pi(s-2t)}{3k}
	=0,
\]
which is valid only if either $s=0$, $t=0$, 
$s=t$ or $s+t\equiv 0\pmod 3$. 
The cases for $s=0$, $t=0$ and $s=t$ 
are excluded because then $w_{s,t}^{\pm}\equiv 0$. 
\end{proof}
\begin{proof}[Proof of 
Theorem~{\upshape\ref{Thm(o(k^2)_evs)}}]
As is easily confirmed, 
$f_{s,t}^{\pm}=u_{s,t}^{\pm}$ or $w_{s,t}^{\pm}$ satisfies
\[
	\begin{aligned}
		f_{s,t}^{\pm}
		={}&
		f_{t,s}^{\pm}
		=f_{t-s,t}^{\pm}
		=f_{t,t-s}^{\pm}
		=f_{3k-s,3k-t}^{\pm}
		=f_{3k-t,3k-s}^{\pm}
		\\
		={}&
		f_{3k-t+s,3k-t}^{\pm}
		=f_{3k-t,3k-t+s}^{\pm}
		=f_{3k-s,t-s}^{\pm}
		=f_{t-s,3k-s}^{\pm}
		=f_{s,3k-t+s}^{\pm}
		=f_{3k-t+s,s}^{\pm}
	\end{aligned}
\]
and therefore, 
by Lemmas~\ref{Lem(all_inv&alt_evs)} 
and \ref{Lem(patapata)}, 
the image of $\iota$ is contained in the vector space, 
say $\mathcal{V}$, 
spanned by $u_{s,t}^{\pm}$'s for
\begin{equation}
	\begin{aligned}
		&
		\left\{
		({\pm},s,t)
		\setmid
		\text{$s+t$ is divisible by $3$},\ 
		0<s<2k,\ \max\{0,2s-3k\}<t<s/2
		\right\}
		\\
		\cup
		&
		\left\{
		({-},2s,s)
		\setmid
		0\leq s<2k
		\right\}
		\\
		\cup
		&
		\left\{
		({+},s,2s-3k)
		\setmid
		3k/2\leq s<2k
		\right\}
		\\
		\cup
		&
		\left\{
		(\pm,s,0)
		\setmid
		\text{$s$ is divisible by $3$},\ 
		0<s<3k/2
		\right\}
	\end{aligned}
	\label{Eq((pm,s,t)_for_u_s,t^pm)}
\end{equation}
and $w_{s,t}^{\pm}$'s for
\begin{equation}
	\begin{aligned}
		&
		\left\{
		({\pm},s,t)
		\setmid
		\text{$s+t$ is divisible by $3$},\ 
		0<s<2k,\ \max\{0,2s-3k\}<t<s/2
		\right\}
		\\
		\cup
		&
		\left\{
		({-},2s,s)
		\setmid
		0\leq s<2k
		\right\}
		\\
		\cup
		&
		\left\{
		({+},s,2s-3k)
		\setmid
		3k/2\leq s<2k
		\right\}.
	\end{aligned}
	\label{Eq((pm,s,t)_for_w_s,t^pm)}
\end{equation}
Since the number of elements of 
(\ref{Eq((pm,s,t)_for_u_s,t^pm)}) 
is given as
\[
	\left\{
	\begin{aligned}
		& \frac{9}{2}j^2+\frac{15}{2}j+3,
		& & \text{if $k=3j+2$},
		\\
		& \frac{9}{2}j^2+\frac{21}{2}j+6,
		& & \text{if $k=3j+3$},
		\\
		& \frac{9}{2}j^2+\frac{27}{2}j+10,
		& & \text{if $k=3j+4$},
	\end{aligned}
	\right.
\]
and that of 
(\ref{Eq((pm,s,t)_for_w_s,t^pm)}) 
is given as
\[
	\left\{
	\begin{aligned}
		& \frac{9}{2}j^2+\frac{9}{2}j+1,
		& & \text{if $k=3j+2$},
		\\
		& \frac{9}{2}j^2+\frac{15}{2}j+3,
		& & \text{if $k=3j+3$},
		\\
		& \frac{9}{2}j^2+\frac{21}{2}j+6,
		& & \text{if $k=3j+4$}
	\end{aligned}
	\right.
\]
($j\geq 0$), total of which is $k^2$ in either case, 
the image of $\iota$ must coincide with $\mathcal{V}$. 
In particular 
the set of $\lambda_{s,t}^{\pm}$'s for 
(\ref{Eq((pm,s,t)_for_u_s,t^pm)}) 
and 
(\ref{Eq((pm,s,t)_for_w_s,t^pm)}) 
is the set of all the eigenvalues of 
the $(k,0)$-cluster. 
\par
The function 
$P(3k)\ni (s,t)\mapsto \lambda_{s,t}^{\mp}\in [0,6]$ 
takes value near $0$ (resp.\ $6$) 
only near the four corners of $P(3k)$. 
The number of vertices among 
(\ref{Eq((pm,s,t)_for_u_s,t^pm)}) 
(resp.\ (\ref{Eq((pm,s,t)_for_w_s,t^pm)})) 
within distance $o(k)$ from the corner 
is $o(k^2)$, which are arbitrarily close to 
$0$ (resp.\ $6$) when $k$ is sufficiently large. 
\end{proof}
\subsection{The case where $X$ is $4$-valent}
The dihedral group $D_4$ of order $8$ 
acts in a natural way on $\mathbb{C}^{\square(k)}$ 
and the notions of \emph{$D_4$-invariant eigenvalue} 
and \emph{$D_4$-alternating eigenvalue} 
are also defined exactly in the same way as in 
$3$-valent case. 
Similarly as in the $3$-valent case, 
we have a well-defined injection
\[
	\iota\colon 
	\mathbb{C}^{\square(k)}
	\rightarrow
	\mathcal{U}_{2k}
	\oplus
	\mathcal{W}_{2k}
\]
which is defined like shown 
in Figure~\ref{Fig(patapata_4vlnt)}, 
where $\mathcal{U}_{2k}$ 
(resp.\ $\mathcal{W}_{2k}$) 
denotes the space of $D_4$-invariant 
(resp.\ $D_4$-alternating 
w.r.t.\ the diagonal line of $\square(2k)$) 
eigenfunctions on the $(2k,0)$-cluster. 
We shall find all the eigenfunctions on $\square(k)$ 
by completely determining the image of $\iota$. 
\begin{figure}[htbp]
	\centering
	\includegraphics[height=5cm]{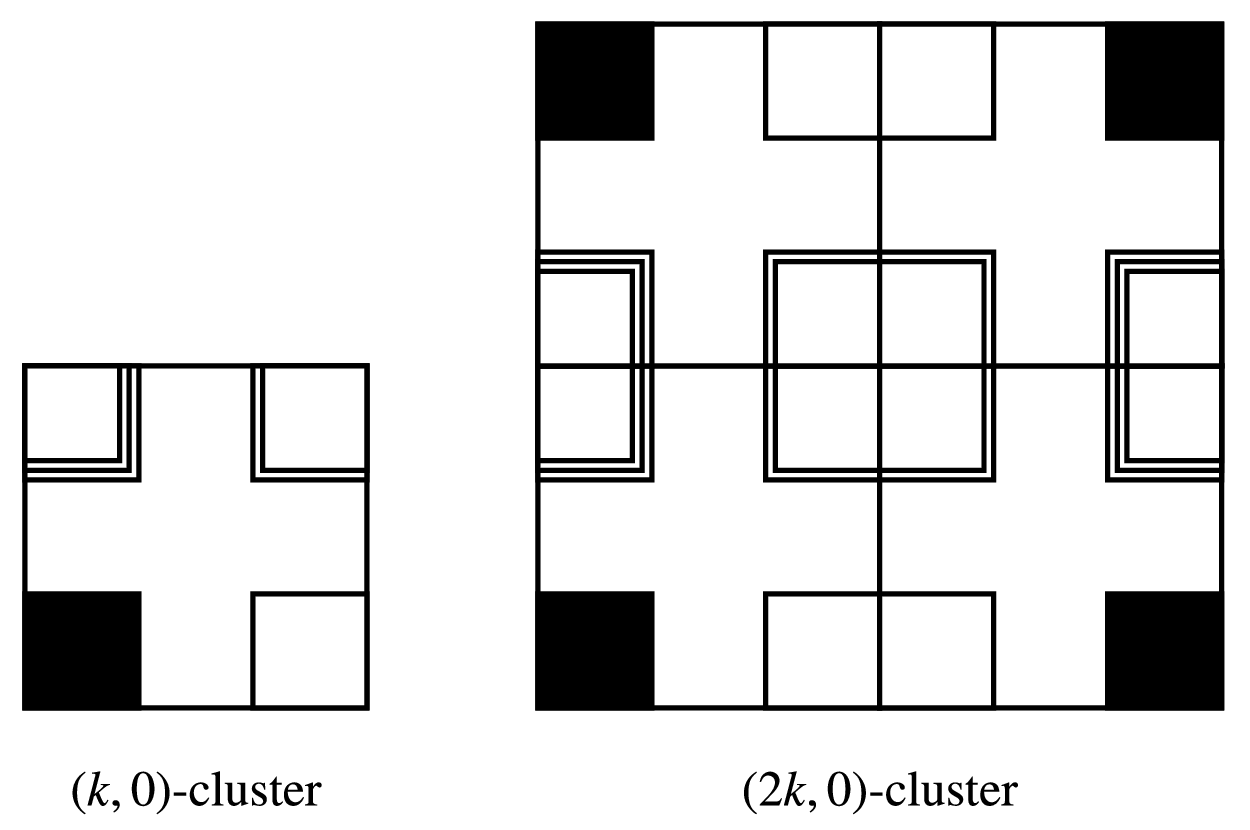}
	\caption{%
	$\square(2k)$ is tiled 
	by the foldings of $\square(k)$.
	}
	\label{Fig(patapata_4vlnt)}
\end{figure}
\par
For $s,t=0,1,\dots,2k-1$,
\begin{equation}
	v_{s,t}(a+bi)
	=
	e^{\pi i(sa+tb)/k},
	\quad
	\text{for 
	$a+bi
	\in
	S(2k)
	:=
	\left\{
	a+bi
	\setmid
	0\leq a,b\leq 2k-1
	\right\}$,}
\end{equation}
give all the eigenfunctions of the \emph{torus} 
which is obtained by adding edges between 
$a$ and $a+(2k-1)i$, and between 
$bi$ and $2k-1+bi$ for each $a,b=0,1,\dots,2k-1$. 
The corresponding eigenvalues are given as
\begin{equation}
	\lambda_{s,t}
	=
	\lambda_{s,t}(2k)
	=
	4-2\cos\frac{\pi s}{k}-2\cos\frac{\pi t}{k}.
	\label{Eq(4-vlnt_lambda_st)}
\end{equation}
Let
\[
	u_{s,t}
	:=
	\sum_{\sigma\in D_4}\sigma v_{s,t}
	\quad
	\text{and}%
	\quad
	w_{s,t}
	:=
	\sum_{\sigma\in D_4}\sgn(\sigma)\sigma v_{s,t}
\]
be the projections of $v_{s,t}$ to 
$\mathcal{U}_{2k}$ and $\mathcal{W}_{2k}$ respectively. 
Note here that, unlike the $3$-valent case, 
$u_{s,t}$ and $w_{s,t}$ always lie in the image of $\iota$. 
Similar computations as in the proof of 
Lemma~\ref{Lem(all_inv&alt_evs)} show that 
$u_{s,t}\equiv 0$ if and only if $u_{s,t}$ is either
\begin{itemize}
	\item 
	$u_{s,k}$ for $0\leq s\leq 2k-1$; or
	\item 
	$u_{k,t}$ for $0\leq t\leq 2k-1$,
\end{itemize}
and that 
$w_{s,t}\equiv 0$ if and only if 
$w_{s,t}$ is one of the following:
\begin{itemize}
	\item 
	$w_{s,k}$ for $0\leq s\leq 2k-1$;
	\item 
	$w_{k,t}$ for $0\leq t\leq 2k-1$;
	\item 
	$w_{s,s}$ for $0\leq s\leq 2k-1$;
	\item 
	$w_{s,2k-s}$ for $0\leq s\leq 2k-1$.
\end{itemize}
Moreover a simple computation shows that 
if both $s$ and $t$ are divisible by $2$, 
then $u_{s,t}|_{\square(k)}$ 
is also a $D_4$-invariant eigenfunction. 
\par
Since $f_{s,t}=u_{s,t}$ or $w_{s,t}$ satisfies
\[
	\begin{aligned}
		f_{s,t}
		={}&
		f_{t,s}
		=e^{-\pi is/k}f_{2k-s,t}
		=e^{-\pi is/k}f_{t,2k-s}
		\\
		={}&
		e^{-\pi it/k}f_{2k-t,s}
		=e^{-\pi it/k}f_{s,2k-t}
		=e^{-\pi i(s+t)/k}f_{2k-s,2k-t}
		=e^{-\pi i(s+t)/k}f_{2k-t,2k-s}
	\end{aligned}
\]
and therefore
\begin{equation}
	\left\{
	u_{s,t}
	\setmid
	t\leq s\leq k-1,\ 
	0\leq t\leq k-1
	\right\}
	\cup
	\left\{
	w_{s,t}
	\setmid
	t+1\leq s\leq k-1,\ 
	0\leq t\leq k-1
	\right\}
	\label{Eq(u_st&w_st)}
\end{equation}
gives a complete list of 
the eigenfunctions of $\square(k)$ 
because its total number is computed as 
$(k^2/2+k/2)+(k^2/2-k/2)=k^2$. 
\par
Theorem~\ref{Thm(o(k^2)_evs)} 
is now proved similarly as in the $3$-valent case 
only by noting that the function 
$S(2k)\ni (s,t)\mapsto \lambda_{s,t}\in [0,8]$ 
takes value near $0$ (resp.\ $8$) 
only near the four corners 
(resp.\ the center) of $S(2k)$.
\begin{proof}[Proofs of\/ 
{\upshape(\ref{Eq(lambda(GC(X))>6-e)})} 
in Theorem~{\upshape\ref{Thm(comp_with_X)}} 
and Theorem~{\upshape\ref{Thm(arbitrary_lambda)}}]
The same computation as above shows 
that the projection $u_{s,t}\in \mathcal{U}_k$ 
of $v_{s,t}$ defined on $S(k)$ 
vanishes if and only if 
$k$ is even and either
\begin{itemize}
	\item 
	$u_{s,k/2}$ for $0\leq s\leq k-1$; or
	\item 
	$u_{k/2,t}$ for $0\leq t\leq k-1$.
\end{itemize}
Therefore
\[
	\lambda_{s,s}(k)
	=
	4-4\cos\frac{2\pi s}{k}
	\quad
	\text{for $0\leq s\leq k-1$ and $s\neq k/2$}
\]
are all $D_4$-invariant eigenvalues 
for the $(k,0)$-cluster, 
and an arbitrary real number in $[0,8]$ 
is approximated by these values as $k$ tends to infinity, 
which proves Theorem~\ref{Thm(arbitrary_lambda)}. 
\par
(\ref{Eq(inv_lambda<lambda(GC(X)))}) is valid also 
for a $4$-valent graph, 
and the inequality (\ref{Eq(lambda(GC(X))>6-e)}) 
is obtained by choosing 
$\lambda=\lambda_{(k-2)/2,(k-2)/2}(k)=4+4\cos(2\pi/k)$ 
if $k$ is even and 
$\lambda=\lambda_{(k-1)/2,(k-1)/2}(k)=4+4\cos(\pi/k)$ 
if $k$ is odd.
\end{proof}
%
\section{On the eigenvalues $2$ and $4$ 
for Goldberg-Coxeter constructions}
\label{Sec(ev_2_4)}
This section provides proofs of the 
theorems on multiplicities of eigenvalues $2$ and $4$ 
stated in Section~\ref{Sec(Intro)}. 
In the first two subsections, 
we shall prove Theorems~\ref{Thm(eigen24_3-vlnt)} 
and \ref{Thm(eigen4_4-vlnt)}. 
As is seen below, a reason for 
large multiplicities of eigenvalues $2$ and $4$ 
of $\GC_{2k,0}(X)$ is that the $(2k,0)$-clusters also 
have large multiplicities of eigenvalues $2$ and $4$. 
On the other hand, 
it is considered that the structure of an initial graph $X$ 
would affect the eigenvalue distribution 
of its Goldberg-Coxeter constructions. 
A few remarkable examples shall be provided 
in Section~\ref{Sec(3m-gons)}, 
where a proof of Theorem~\ref{Thm(3m-gons)} 
is also included. 
\subsection{The case where $X$ is $3$-valent}
From what was mentioned 
in (1) of Remark~\ref{Rem(inv&alt)}, 
Theorem~\ref{Thm(eigen24_3-vlnt)} 
is an immediate consequence of 
the following lemma.
\begin{Lem}
\label{Lem(eigen24_3-vlnt_clstr)}
For $k\geq 1$ {\upshape(}resp.\ $k\geq 2${\upshape)}, 
the $3$-valent $(2k,0)$-cluster $\triangle(2k)$ 
has a $D_3$-invariant eigenvalue 
$4$ {\upshape(}resp.\ $2${\upshape)}, 
whose multiplicity is at least 
$\lceil k/2\rceil$ 
{\upshape(}resp.\ $\lfloor k/2\rfloor${\upshape)}. 
\end{Lem}
\begin{proof}
For $0\leq s,t<6k$, 
as is easily proved from a direct computation 
using (\ref{Eq(3-vlnt_lambda_st^pm)}), 
$\lambda_{s,t}^{+}(6k)$ (resp.\ $\lambda_{s,t}^{-}(6k)$) 
takes the value $4$ (resp.\ $2$) if and only if 
$s$ and $t$ satisfy either 
$s=3k$ or $t=3k$ or $s-t=3k$ or $s-t=-3k$. 
Among (\ref{Eq((pm,s,t)_for_u_s,t^pm)}) 
with $k$ replaced with $2k$, 
$\lambda_{s,t}^{+}=4$ 
for, and only for, $s=3k$ and $t=3j$ ($0\leq j<k/2$), 
and $\lambda_{s,t}^{-}=2$ 
for, and only for, $s=3k$ and $t=3j$ ($0<j\leq k/2$). 
All of them are $D_3$-invariant 
by Lemma~\ref{Lem(patapata)}. 
The corresponding $u_{s,t}^{\pm}$'s 
are linearly independent 
from the consequence obtained in the proof of 
Theorem~\ref{Thm(o(k^2)_evs)}. 
\end{proof}
\begin{figure}[htbp]
	\centering
	\begin{tabular}{cc}
		\subfigure[with eigenvalue $4$]%
		{\includegraphics[width=.35\textwidth]%
		{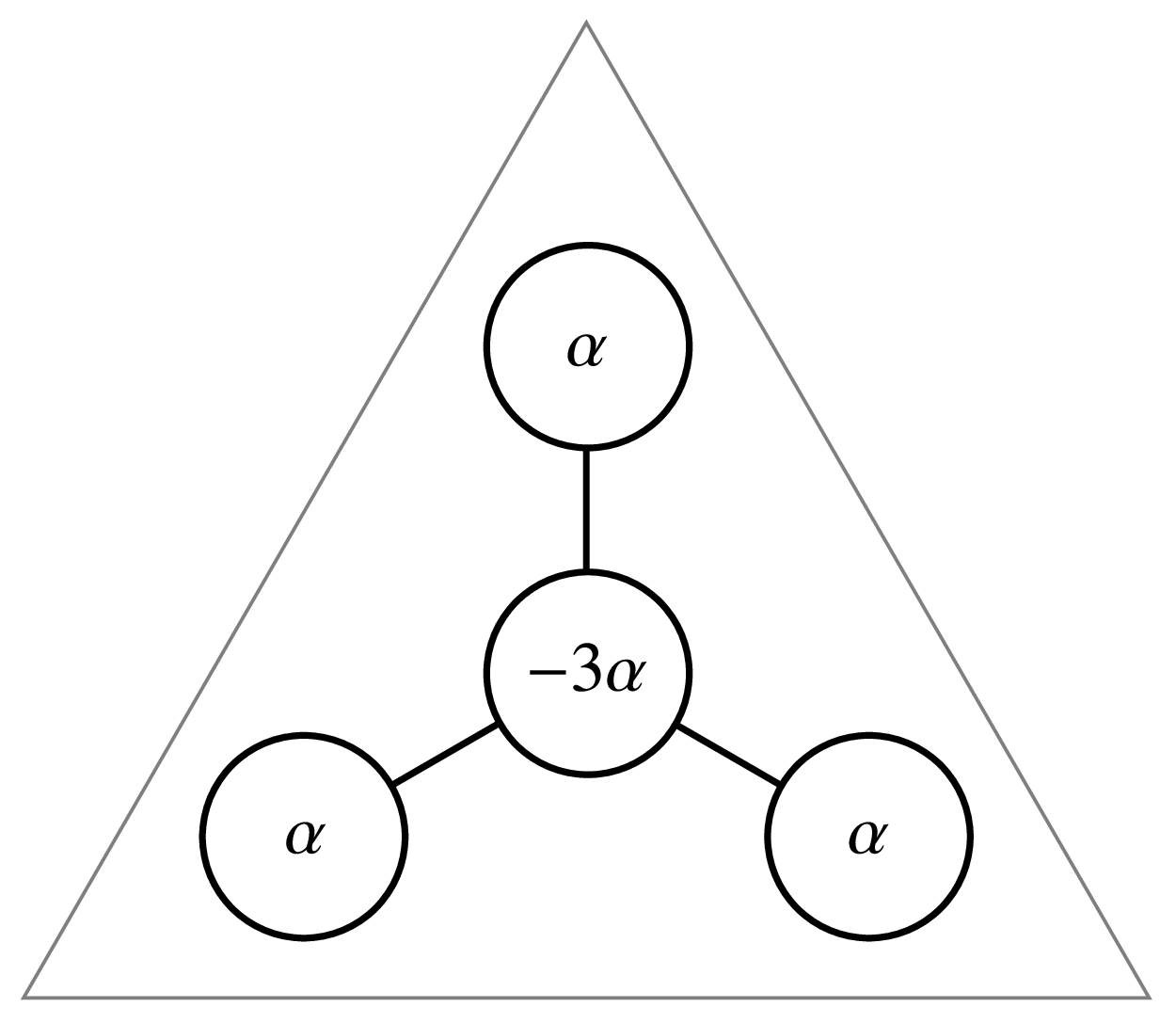}%
		\label{Fig(eigen4)}}
		&
		\subfigure[with eigenvalue $2$]%
		{\includegraphics[width=.35\textwidth]%
		{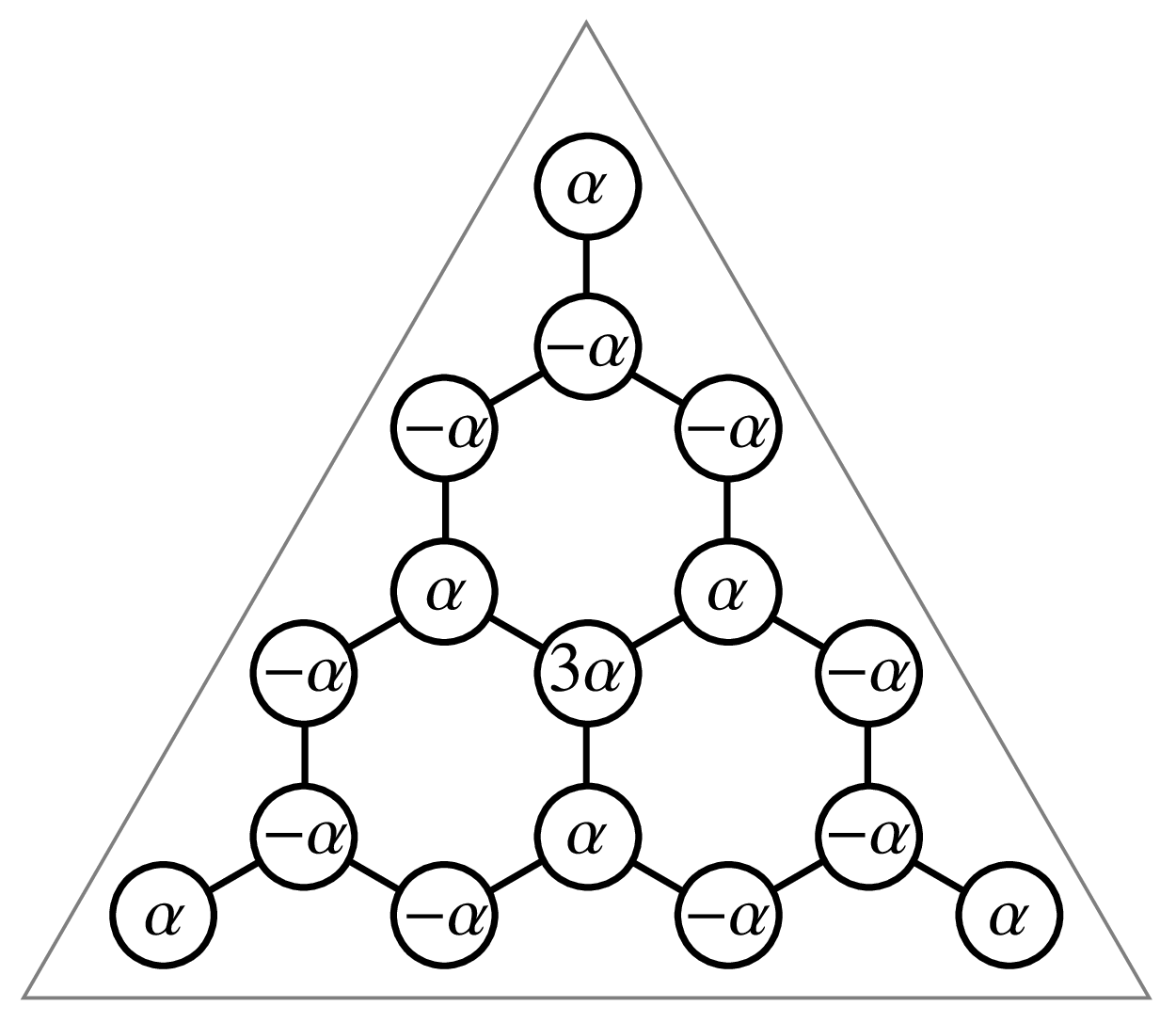}%
		\label{Fig(eigen2)}}
	\end{tabular}
	\caption{$D_3$-invariant eigenfunctions}
	\label{Fig(cluster)}
\end{figure}
\subsection{The case where $X$ is $4$-valent}
Similarly as in the $3$-valent case, 
Theorem~\ref{Thm(eigen4_4-vlnt)} 
is a consequence of the following.
\begin{Lem}
\label{Lem(eigen4_4-vlnt_clstr)}
For $k\geq 2$, 
the $4$-valent $(2k,0)$-cluster $\square(2k)$ 
has a $D_4$-invariant eigenvalue $4$, 
whose multiplicity is at least $\lceil (k-1)/2\rceil$. 
\end{Lem}
\begin{proof}
It is easily confirmed that $\lambda_{s,t}(4k)=4$ 
if and only if $s$ and $t$ satisfy 
either $s+t=2k$ or $s+t=6k$ or $s-t=2k$ or $s-t=-2k$. 
The number of $u_{s,t}$'s 
out of (\ref{Eq(u_st&w_st)}) satisfying 
$\lambda_{s,t}(4k)=4$ is therefore computed as 
$\lceil (k-1)/2\rceil$. 
\end{proof}
\subsection{Dependence on the structure of $X$ 
(only for $3$-valent case)}
\label{Sec(3m-gons)}
This subsection 
provides the proofs of Theorem~\ref{Thm(3m-gons)}, 
which describes relations between the conditions 
(F), (CN), (N) and (C) 
in Section~\ref{Sec(conditions_on_3-vlnt_graphs)}, 
and eigenvalues $2$ and $4$ of some $\GC_{k,0}(X)$'s. 
\par
\begin{proof}[Proof of\/ {\upshape(1)} 
of Theorem~{\upshape{\upshape\ref{Thm(3m-gons)}}}]
Let us take a vertex numbering 
$V(\GC_{2,0}(X))\rightarrow \{1,2,3\}$ 
satisfying (N), 
whose existence is guaranteed by 
Proposition~\ref{Prop(coherent_numbering)}. 
Let $\alpha_0,\alpha_1,\alpha_2$ and $\alpha_3$ be 
complex numbers satisfying 
$\alpha_0+\alpha_1+\alpha_2+\alpha_3=0$. 
Then it is easy to see that the function 
$v\colon V(\GC_{2,0}(X))\rightarrow \mathbb{C}$ 
which maps a vertex with number $i$ to $\alpha_i$ 
is an eigenfunction of $\Delta_{\GC_{2,0}(X)}$ 
with eigenvalue $4$. 
By above reasons, we can find two more eigenfunctions 
which are linearly independent with $u=u_{\alpha}$ 
which was obtained in 
Theorem~\ref{Thm(eigen24_3-vlnt)} (1); 
in fact, set 
$(\alpha_0,\alpha_1,\alpha_2,\alpha_3)
=(0,1,-1,0),(0,1,1,-2)$ for example. 
\end{proof}
Let us next consider the condition (C). 
Our assertions are summarized as follows. 
\begin{Prop}\label{Prop(C_and_mult)}
Let $X$ be a $3$-valent plane graph. 
\begin{enumerate}
	\item 
	If $X$ has a vertex coloring satisfying 
	{\upshape(C-i)} and {\upshape(C-ii)}, 
	then for any $s\in \mathbb{N}$, 
	$\GC_{2s-1,0}(X)$ has eigenvalue $4$. 
	\item 
	If $X$ has a vertex coloring satisfying 
	{\upshape(C-i)}, {\upshape(C-ii)} and {\upshape(C-iii)}, 
	then for any $k\in \mathbb{N}$, 
	both $\GC_{k,0}(X)$ and $\GC_{k,k}(X)$ 
	have eigenvalue $4$ {\upshape(}resp.\ $2${\upshape)}, 
	whose multiplicity is at least $\lceil k/2\rceil$ 
	{\upshape(}resp.\ $\lfloor k/2\rfloor${\upshape)}. 
\end{enumerate}
\end{Prop}
\begin{proof}
(1) The function $u\colon V(X)\rightarrow \mathbb{C}$ 
which maps a black vertex to $-3$ 
and a white one to $1$ 
is an eigenfunction of $\Delta_X$ with eigenvalue $4$, 
which proves (1) for $s=1$. 
\par
For $s\geq 2$, a quadruplet 
$\{\triangle(p),\triangle(q_1),
\triangle(q_2),\triangle(q_3)\}$ 
of $(2s-1,0)$-clusters, 
where $p$ is black and 
$N_X(p)=\{q_1,q_2,q_3\}$ are all white, 
can be glued with each other to be identified with 
a $(4s-2,0)$-cluster. 
On the other hand, 
it follows from a direct computation that 
$u_{3(2s-1),0}^{+}\in \mathcal{U}_{3(4s-2)}$ 
of (\ref{Eq(u_st^pm&w_st^pm)}) 
gives a $D_3$-invariant eigenfunction 
on $\triangle(4s-2)$ with the eigenvalue $4$ 
\emph{with the constant boundary value} $4$.
Therefore $u$ defines an eigenfunction on $\GC_{2s-1,0}(X)$ 
with eigenvalue $4$, proving (1). 
\par
(2) In the argument above 
to prove the existence 
$u$ on $\triangle(4s-2)$, 
if (C-iii) is further satisfied, 
then \emph{any} $D_3$-invariant eigenfunction 
on $\triangle(4s-2)$ with eigenvalue $4$ 
(resp.\ eigenvalue $2$) gives an eigenfunction 
on $\GC_{2s-1,0}(X)$ 
with eigenvalue $4$ (resp.\ eigenvalue $2$). 
For exactly the same reason, 
any $D_3$-invariant eigenfunction on $\triangle(4s)$ 
with eigenvalue $4$ (resp.\ eigenvalue $2$) 
gives an eigenfunction $\GC_{2s,0}(X)$ 
with eigenvalue $4$ (resp.\ eigenvalue $2$). 
This and 
(\ref{Item(GC_1,1_satisfies_C)}) 
in Examples~\ref{Ex(F-CN-C)} 
now prove (2). 
\end{proof}
\begin{proof}[Proof of {\upshape(2)} 
of Theorem~{\upshape{\upshape\ref{Thm(3m-gons)}}}]
The assertion 
is an immediate consequence of 
Proposition~\ref{Prop(C_and_mult)} 
and 
Proposition~\ref{Prop(F_implies_C)}. 
\end{proof}
\subsubsection*{Acknowledgment} 
Authors are partially supported by 
JSPS KAKENHI Grant Number 
15K17546, 15H02055, 25400068, 18K03267, 26400067, 17H06465, 17H06466, and 19K03488. This work is also supported 
by JST CREST Grant Number JPMJCR17J4. 
\end{document}